\documentclass[11pt]{amsart}
\usepackage{amsmath}
\usepackage{amsfonts}
\usepackage{amssymb}
\usepackage{amscd}
\usepackage{epsfig}

\catcode `\@=11
\def\numberbysection{\@addtoreset{equation}{section}
         \renewcommand{\theequation}{\thesection.\arabic{equation}}}
\numberbysection
\def\subsubsection{\@startsection{subsubsection}{3}%
  \normalparindent{.5\linespacing\@plus.7\linespacing}{-.5em}%
  {\normalfont\bfseries}}
\newlength{\ghost}

\newtheorem{thm}{Theorem}[section]
\newtheorem{lem}[thm]{Lemma}
\newtheorem{prop}[thm]{Proposition}
\newtheorem{cor}[thm]{Corollary}

\newtheorem{conj}[thm]{Conjecture}

\theoremstyle{definition}

\newtheorem{df}[thm]{Definition}
\newtheorem{ex}[thm]{Example}

\newtheorem{rmk}[thm]{Remark}
\newtheorem{nota}[thm]{Notation}
\newtheorem{assump}{Assumption}

\def\Sn{\mathbb{S}_n}

\def\Z{\mathbb{Z}}

\def\t{\tau}
\def\a{\alpha}
\def\b{\beta}
\def\g{\gamma}

\def\Rp{{\mathbb{R}_{>0}}}

\def\O{\mathcal{O}}

\def\Cact{\mathcal{C}act}
\def\Cacti{\mathcal{C}acti}

\def\Arc{\mathcal{A}rc}

\def\Gtree{{\mathcal{GT}ree}}
\def\GTree{\Gtree}
\def\Sttree{{\mathcal S}t\Gtree}
\def\LGtree{{\mathcal{LGT}ree}}
\def\CGtree{{\mathcal{CGT}ree}}
\def\Lintree{{\mathcal{LT}ree}}
\def\Tree{{\mathcal{T}ree}}
\def\St{{\mathcal S}t}
\def\Fat{{}^{\rm Fat}}
\def\T{T}
\def\H{H}
\def\G{G}

\def\N{{\mathbb N}}

\def\R{{\mathbb R}}

\newcommand\val[1]{\text{\it val}(#1)}

\def\Agrs{\A_{g,r}^{s}}
\def\comp{Comp}
\def\graphs{\mathcal{G}}
\def\A{\mathcal{A}}
\def\arcgraphs{\overline{\mathcal{G}}}
\def\carcgraphs{\arcgraphs^{e}}

\def\ba{\bar{\a}}
\def\bg{\bar{\g}}

\def\Arcn{{\mathcal{A}rc_{\#}}}
\def\Darc{\mathcal{DA}rc}

\def\SCC{\mathcal{S}CC}
\def\Corol{\mathcal{C}orol}

\def\seq{seq}
\def\defect{dt}
\def\edefect{\epsilon}
\def\sth{st_H}
\def\stg{st_G}

\def\colim{{\rm colim}}

\def\K{{\mathcal K}}

\def\MS{{\mathcal MS}}
\def\calS{{\mathcal S}}
\def\calO{{\mathcal O}}
\def\calD{{\mathcal D}}

\def\cell{\overline {C}}

\def\gap{gap}
\begin{document}

\title[Dimension vs.\ genus]
{Dimension vs.\ Genus: A surface realization of the little
$k$--cubes and an $E_{\infty}$--operad}

\author
[Ralph M.\ Kaufmann]{Ralph M.\ Kaufmann}
\email{rkaufman@math.purdue.edu}

\address{Purdue University, Department of Mathematics, 150 N. University St.,
West Lafayette, IN 47907--2067}

\begin{abstract}
We define a new $E_{\infty}$ operad  based on surfaces
with foliations which contains $E_k$ suboperads. We construct CW
models for these operads and provide applications of these models by
giving actions on Hochschild complexes --thus making contact with
string topology--, by giving explicit cell representatives
 for the Dyer-Lashof-Cohen operations for the
2-cubes and by constructing new $\Omega$ spectra. The underlying
novel principle is that we can trade genus in the surface
representation vs.\ the dimension $k$ of the little $k$--cubes.
\end{abstract}

\maketitle

\section*{Introduction}
The fact \cite{cact}   that the cacti operad introduced in
\cite{vor} has an $E_2$ suboperad  has been instrumental for the
considerations of string topology \cite{CS,Sul}. In terms of
algebraic actions this particular $E_2$ operad has been useful in
describing actions on the Hochschild cochains of an associative
algebra \cite{del}. All these considerations have some form of
physical 1+1 dimensional field theoretical inspiration or
interpretation, which for a mathematician essentially means that one
is dealing with maps of surfaces. In particular the $E_2$ structure
of the little discs and cacti is at home in such a 2--dimensional
geometry.

In this context, the natural question arises if the higher order
$E_k$ operads can also be realized on surfaces. According to the
yoga of string theory, two dimensional structures should be enough.
In particular one should be able to describe higher dimensional
objects, like branes, with strings. In our setting this translates
to the expectation that there should be surface realizations for
$E_k$ operads. The fulfillment of this expectation is exactly what
we accomplish. The novel feature is that these surfaces are of
arbitrary genus and not only of genus zero. Now, as soon as one
introduces genus into an operadic structure, the genus ceases to be
bounded. This is why we first construct an $E_{\infty}$ operad using
surfaces with boundaries of all genera. The way we identify
$E_{\infty}$ structure is to invoke Berger--Fiedorowicz's theory
\cite{Berger,Fiedobscure} of $E_n$ and $E_{\infty}$ operads. Hence
we obtain a filtration of our $E_{\infty}$ operad by $E_n$ operads.
This filtration is roughly by genus and it exhibits  an interesting
periodicity. The $E_{2k}$ and $E_{2k+1}$ operads are both realized
basically by genus $k$ surfaces with boundaries. More precisely,
their operadic degree $2$ components are realized on a surface of
the indicated genus. The specific difference between the operads
$E_{2k}$ and $E_{2k+1}$ is identified
 to be the possibility to twist on one particular boundary, namely
the ``out boundary''.

The method we use for the construction is based on the $\Arc$
operad, whose formalism we briefly review. Just like for cacti there
will be a certain tree condition. Since although arc graphs are for
us the most natural language, the language of ribbon graphs is more
widely used, we provide an Appendix with a dual description in terms
of ribbon graphs. If one so wishes using this dictionary one can
translate all the results into this purely combinatorial language
thereby sacrificing their geometric origin.

In order to produce the operads, we will have to use a new technique
of ``stabilizing''. It is clear that some identifications have to be
made, since we know from representations or better algebras over the
operad $H_*(E_2)$ that the Gerstenhaber bracket does not always
vanish and likewise neither does the string bracket. In other words
we should not be able to find a homotopy which ``kills off'' the
cell for the bracket in the usual formalism of arcs and moduli
spaces \cite{hoch1,hoch2}. In fact, for a Frobenius algebra, we know
from \cite{hoch2} that the obstruction to ``kill'' the bracket is
the Euler element of the algebra. On the geometric level we can
force the homotopy, by identifying boundary components of a cellular
compactification with cells comprised of lower genus surfaces. This
 is what our stabilization procedure formalizes. In
the algebraic setting this stabilization is possible in the case
that the algebra is semi--simple and has a particularly simple
metric.

Our constructions can be generalized to the full $\Arc$ spaces and
will yield a new way to stabilize moduli spaces. In the future, we
also expect to find explicit formulas for the higher
Dyer--Lashof--Cohen operations using the new geometric insights from
the surface formalism.

The paper is organized as follows: In the first section, we review
the basic setup for the arc operad to make the paper more self
contained. The second section introduces the stabilization for the
various tree operads. The straightforward construction yields
operads without a $0$--term --just like  cacti. In order to obtain a
$0$--term for the operad we consider a thickening of the
construction. Without the thickening we can add a $0$--term, but
then the associativity will hold only up to homotopy. In the third
section, we show that stabilization and thickening yield a cellular
$E_{\infty}$ operad in the sense of Berger. There is a filtration on
the $E_{\infty}$ operad giving rise to a new surface representation
for  $E_k$ operads. Without the $0$--terms we can omit the
thickening step. The fourth section passes to the chain level and
gives cellular models as well as operations, such as $\cup_i$ and
the Dyer--Lashof--Cohen operations. For the chain level, as we show,
one can omit the thickening procedure as the induced structure of
the stabilization is already an operad even if one includes a
$0$--term. The last section contains applications to string topology
and Hochschild actions as well as the construction of a new $\Omega$
spectrum. We also discuss the generalizations to the $\Arc$ operad
and the Sullivan PROP. The Appendix contains the dual ribbon graph
picture.

\section*{Acknowledgments}
We wish to thank the organizers of the Postinkov Memorial
conference. This was a wonderful place to share ideas and gain new
insights. We also thank the editors of this volume for encouraging
us to stop procrastinating and finally write down these results that
had long been announced.

This paper owes its existence to discussions which have been carried
on over by now some years with Jim McClure whom it is a pleasure to
thank. We also wish to thank Craig Westerland for discussions on his
work and Paolo Salvatore for comments.
Some of the research was carried out while we were visiting
the Max--Planck--Institute in Bonn and we gratefully acknowledge its
support.

\section{Reviewing the $\Arc$ operad}

\label{arcsection} In order to be more self--contained, we begin
with reviewing the constructions of the $\Arc$ operad of \cite{KLP}.
Ultimately we will specialize to a suboperad in this paper, but the
gluing procedure is of course still the same. Furthermore, the more
general point of view will allow for some generalizations in the
future. We will follow \cite{hoch1} for this abbreviated exposition.
The reader familiar with these constructions may skip ahead only
consulting \S\ref{subspaces} for the definition of the new
suboperads we will consider.

\subsection{Spaces of graphs on surfaces}
Fix an oriented surface $F_{g,r}^s$ of genus $g$ with $s$ punctures
and $r$ boundary components which are labeled from $0$ to $r-1$,
together with marked points on the boundary, one for each boundary
component. We call this data $F$ for short if no confusion can
arise.

The piece of the $\Arc$ operad supported on $F$ will be an open
subspace of a space $\Agrs$. The latter space is a CW complex whose
cells are indexed by graphs on the surface $F_{g,r}^s$ up to the
action of the pure mapping class group PMC which is the group of
orientation preserving homeomorphisms of $F_{g,r}^s$ modulo
homotopies that pointwise fix  the set which is the union of the set
of the marked points on the boundary and the set of punctures. A
quick review in terms of graphs is as follows.

\subsubsection{Embedded Graphs}
By an embedding of a graph $\Gamma$ into a surface $F$, we mean an
embedding  $i:|\Gamma|\rightarrow F$ with the conditions

 \begin{itemize}
\item[i)] $\Gamma$ has at least one edge.

\item[ii)] The vertices map bijectively to the marked points on
the boundaries.

\item[iii)] No images of two edges are homotopic to each other, by homotopies
fixing the endpoints.

\item [iv)] No image of an edge is homotopic to a part of the
boundary, again by homotopies fixing the endpoints.
\end{itemize}

Two embeddings are equivalent if there is a homotopy of embeddings
of the above type  from one to the other. Note that such a
homotopy is necessarily constant on the vertices.

The images of the edges are called arcs, and the set of connected
components of $F\setminus i(\Gamma)$ are called complementary
regions.

Changing representatives in a class  yields natural bijections of
the sets of arcs and connected components of $F\setminus
i(\Gamma)$ corresponding to the different representatives.
 We can therefore associate to each equivalence class of embeddings
its sets of arcs together
with their incidence conditions and connected
components --- strictly speaking of course the equivalence classes of these
objects.

\begin{df}
By a graph $\g$ on a surface we mean a triple $(F,\Gamma,[i])$ where
$[i]$ is an equivalence class of embeddings of $\Gamma$ into that
surface. We will denote the isomorphism class of the set of
complementary regions by $\comp(\g)$. We will also set
$|\g|=|E_{\Gamma}|$, were $E_{\Gamma}$ is the set of edges of
$\Gamma$. Fixing the surface $F$, we will call the set of graphs on
a surface $\graphs(F)$.
\end{df}

\subsubsection{A linear order on arcs}
Notice that due to the orientation of the surface the graph
inherits an induced linear order of all the flags at every vertex
$F(v)$ from the embedding. Furthermore there is even a linear
order on all flags by enumerating the flags first according to the
boundary components on which their vertex lies and then according
to the linear order at that vertex. This induces a linear order on
all edges by enumerating the edges by the first appearance of a
flag of that edge.

\subsubsection{The poset structure}
The set of such graphs on a fixed surface $F$ is a poset. The
partial order is given by calling $(F,\Gamma',[i'])\prec
(F,\Gamma,[i])$ if $\Gamma'$ is a subgraph of $\Gamma$ with the
same vertices and $[i']$ is the restriction of $[i]$ to $\Gamma'$.
In other words, the first graph is obtained from the second by
deleting some arcs.

 We
associate a simplex $\Delta(F,\Gamma,[i])$ to each such graph.
$\Delta$ is the simplex whose vertices are given by the set of
arcs/edges enumerated in their linear order. The face maps are
then given by deleting the respective arcs. This allows us to
construct a CW complex out of this poset.

\begin{df}
Fix $F=F_{g,r}^s$. The space $\A_{g,r}^{\prime s}$ is the space
obtained by gluing the simplices $\Delta(F,\Gamma',[i'])$ for all
graphs on the surface according to the face maps.
\end{df}

The pure mapping class group naturally acts on $\A_{g,r}^{\prime s}$
and has finite isotropy \cite{KLP}.

\begin{df}
We let $\A_{g,r}^s:= \A_{g,r}^{\prime s}/PMC$ be the quotient space
and call its elements arc families.
\end{df}

\subsubsection{CW structure of $\A_{g,r}^s$}

\begin{df}
Given a graph on a surface, we call its PMC orbit its arc graph. If
$\g$ is a graph on a surface, we denote by $\bar{\g}$ its arc graph
or PMC orbit. We denote the set of all arc graphs of a fixed surface
$F$ by $\arcgraphs(F)$. A graph is called exhaustive if there are no
vertices $v$ with valence $\val{v}=0$. This condition is invariant
under $PMC$ and hence we can speak about exhaustive arc graphs.  The
set of all exhaustive arc graphs on $F$ is denoted by
$\carcgraphs(F)$.
\end{df}

Since the incidence conditions  are preserved, we can set
$|\bar{\g}|=|\g|$ where $\g$ is any representative and likewise
define $\comp(\bar{\g})$. We call an arc graph
 exhaustive if and only if it contains no isolated vertices, that is
 vertices with $\val{v}=0$.

Now by construction  it is clear that $\Agrs$ is realized as a CW
complex which has one cell  of
dimension $|\g|-1$ for each arc graph $\bar{\g}$. Moreover the cell for a given class of graphs is
actually a map of a simplex whose vertices correspond to the arcs in
the order discussed above. The attaching maps are given by deleting
edges and identifying the resulting face  with its image. Due to the
action of $PMC$ some of the faces  may become identified by these
maps, so that the image will not necessarily be a simplex. The open
part of the cell will however be homeomorphic to an open simplex,
which can be taken as one of its preimages. The PMC action acts
on the graph and hence acts simplicially. Let $C(\ba)$ be
the image of the cell and $\dot C(\ba)$ be its interior, then
\begin{equation}
\A_{g,r}^s=\cup_{\ba\in \arcgraphs(F_{g,r}^s)} C(\ba),\quad
\A_{g,r}^s=\amalg_{\ba\in \arcgraphs(F_{g,r}^s)} \dot C(\ba)
\end{equation}
 Let $\Delta^n$ denote the standard
$n$--simplex and $\dot \Delta$ its interior then
 $\dot C(\g)=\mathbb R^{|E_{\G}|}_{>0}/\mathbb{R}_{>0}=\dot
\Delta^{|E_{\Gamma}|-1}=:C(\Gamma)$ which only depends on the underlying graph
$\Gamma$ of $\g$.
 This also  means that the space $\Agrs$ is
filtered by the cells of dimension less than or equal to $k$.

\subsubsection{Elements of the $\Agrs$ as projectively weighted graphs}
Using barycentric coordinates for the open part of the cells the
elements of $\Agrs$ are given by specifying an arc graph together
with a map $w$ from the edges of the graph $E_{\Gamma}$ to
$\mathbb{R}_{>0}$ assigning a weight to each edge s.t.\ the sum of
all weights is $1$.

Alternatively, we can regard the map $w:E_{\Gamma}\rightarrow
\mathbb{R}_{>0}$ as an equivalence class under the equivalence
relation of, i.e.\ $w\sim w'$ if $\exists \lambda \in
\mathbb{R}_{>0}$ s.t.~ $\forall e\in E_{\Gamma}:   w(e)=\lambda w'(e)$.
That is $w$ is a projective metric. We call the set of $w(e)$ the
projective weights of the edges.
 In the limit, when the projective weight of an
edge goes to zero, the edge/arc is deleted, see \cite{KLP} for more
details. For an example, see Figure \ref{f02delta}, which is
discussed below in Example \ref{arcex}.

An  element $\alpha\in \Agrs$ can be described by a tuple
$\alpha=(F,\Gamma,\overline{[i]},w)$ where $F$ and $\Gamma$ are as
above, $\overline{[i]}$ is a PMC orbit of an equivalence class of
embeddings and $w$ is a projective metric for $\Gamma$. Alternatively it
can be described by a tuple $(\bg,w)$ where $\bg\in \arcgraphs(F)$
and $w$ is a projective metric for the underlying abstract graph
$\Gamma$.
\begin{ex}
\label{arcex} $\A_{0,2}^0=S^1$. Up to PMC there is a unique graph
with one edge and a unique graph with two edges. The former gives a
zero--cell and the latter  gives a one--cell whose source is a
1--simplex. Its two subgraphs with one edge that correspond to the
boundary lie in the same orbit of the action of PMC and thus are
identified to yield $S^1$. The fundamental cycle is given by
$\Delta$ of Figure \ref{f02delta}. Identifying $S^1$ with $\R/\Z$ we
define  $\T_s$ to be  the element corresponding to $s\in S^1$ as
depicted in  Figure \ref{f02delta}.
\end{ex}

\begin{figure}
\epsfxsize = 4in \epsfbox{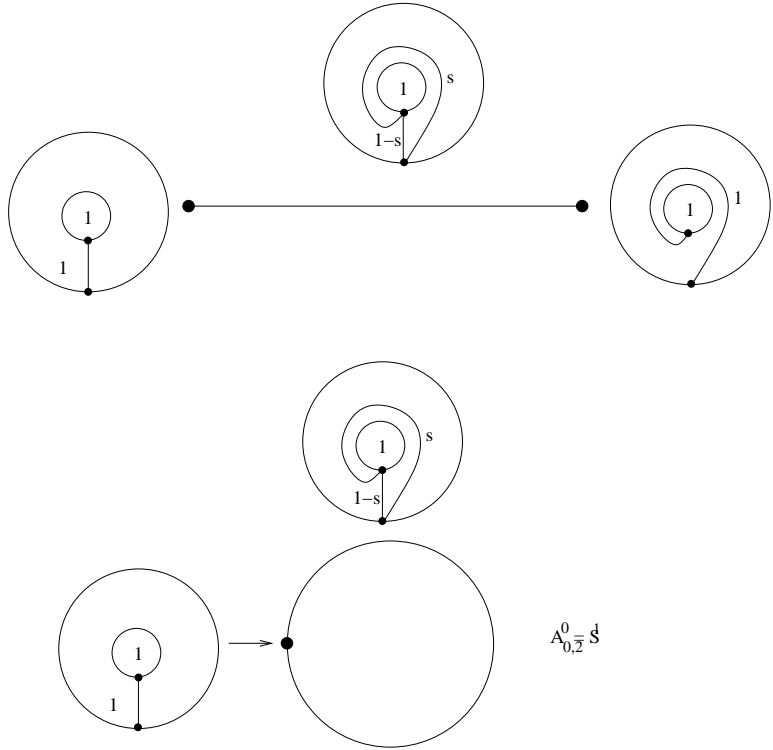} \caption{\label{f02delta}
The space $\A^0_{0,2}$ is given as the CW decomposition of $S^1$
with one $0$--cell and one $1$--cell. It can be thought of as the
quotient of the interval in which the endpoints are identified by
the action of the pure mapping class group. The generator of $CC_*(S^1)$ is
called $\Delta$.}
\end{figure}

\subsubsection{Drawing pictures for Arcs.}
\label{picturepar}
There are several
pictures one can use to view elements of $\A$. In order to draw
elements  it is useful to expand the marked point on the boundary to
an interval called window, and let the arcs end on this interval
according to the linear order. Equivalently, one can mark one point
of the boundary and let the arcs end in their linear order anywhere
{\em but} on this point. We will mostly depict arc graphs in the
latter manner. See Figure \ref{different} for an example of an arc
graph ---all arcs running to the marked points--- and its alternate
depiction with {\em none} of the arcs hitting the marked point and
all arcs having {\em disjoint endpoints}.

\begin{figure}
\epsfxsize = 3in
\epsfbox{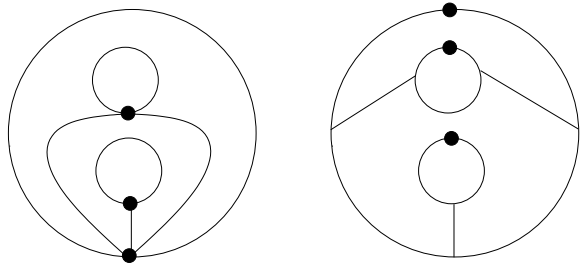}
\caption{\label{different}
An arc graph and its alternate depiction with disjoint arcs not hitting the marked points on the boundary.}
\end{figure}

\begin{nota}
Since in the following we will always be dealing with arc graphs,
we will now omit the over-line in the notation. Hence we will
write $\g\in \arcgraphs(F)$. We also fix that $\Gamma(\g)$ is the
underlying graph. Furthermore elements of $\A_{g,r}^s$ will usually
be called $\a$ and $\b$. If $\a\in \A_{g,r}^s$ we fix that
$\g(\a)$, $\Gamma(\a)$ and $w(\a)$ are the underlying arc graph, its
underlying graph and the projective metric, respectively.
\end{nota}

\subsection{Topological operad structure}
\label{gluingpar}
\subsubsection{The spaces $\Arc(n)$}

\begin{df}
We define $\Arc_g^s(n)\subset \A_{g,n+1}^s$ to be the subset of
those weighted arc graphs whose arc graph is exhaustive. We define
$\Arc(n):=\coprod_{s,g\in \mathbb{N}}\Arc_g^s(n)$.
\end{df}

\subsubsection{Topological description of the gluing } We shall only
give a short recap. The full details are in \cite{KLP}.
To give the composite $\alpha \circ_i\alpha'$ for two arc families
$\alpha=(F,\Gamma,\overline{[i]},w)\in \Arc(m)$ and
$\alpha'=(F',\Gamma',\overline{[i']},w') \in \Arc(n)$ one most
conveniently chooses metrics on $F$ and $F'$. The construction does
not depend on the choice. With this metric, one produces a partially
measured foliation in which the arcs are replaced by bands of
parallel leaves (parallel to the original arc) of width given by the
weight of the arc. For this we choose the window representation and
also make the window tight in the sense that there is no space
between the bands and between the end-points of the window and the
bands. Finally, we put in the separatrices. The normalization we
choose is that the sum of the weights at boundary $i$ of $\alpha$
coincides with the sum of the weights at the boundary $0$, we can
also fix them both to be one. Now when gluing the boundaries, we
match up the windows, which have the same width, and then just glue
the foliations. This basically means that we glue two leaves of the
two foliations if they end on the same point.  We then delete the
separatrices.  Afterwards,  we collect together all parallel leaves
into one band. In this procedure, some of the original bands might
be split or ``cut'' by the separatrices. We assign to each band one
arc with weight given by the width of the consolidated band. If arcs
occur, which do not hit the boundaries, then we simply delete these
arcs. We call these arcs or bands ``closed loops'' and say that
``closed loops appear in the gluing''.

Notice that after gluing there will be no parallel arcs, since all
parallel leaves are collected into {\em one} band and the condition
of being parallel is PMC invariant ---before and after gluing.

\begin{thm}\cite{KLP}
Together with the gluing operations above, the spaces $\Arc$ form
a cyclic operad.
\end{thm}

Another way to see the gluing is in terms of duplicating arcs and
gluing the complementary regions. The duplication occurs when
inserting the separatrices or equivalently cutting the bands; see
Figure \ref{gluingex1} for an example.

\begin{figure}
\epsfxsize = \textwidth \epsfbox{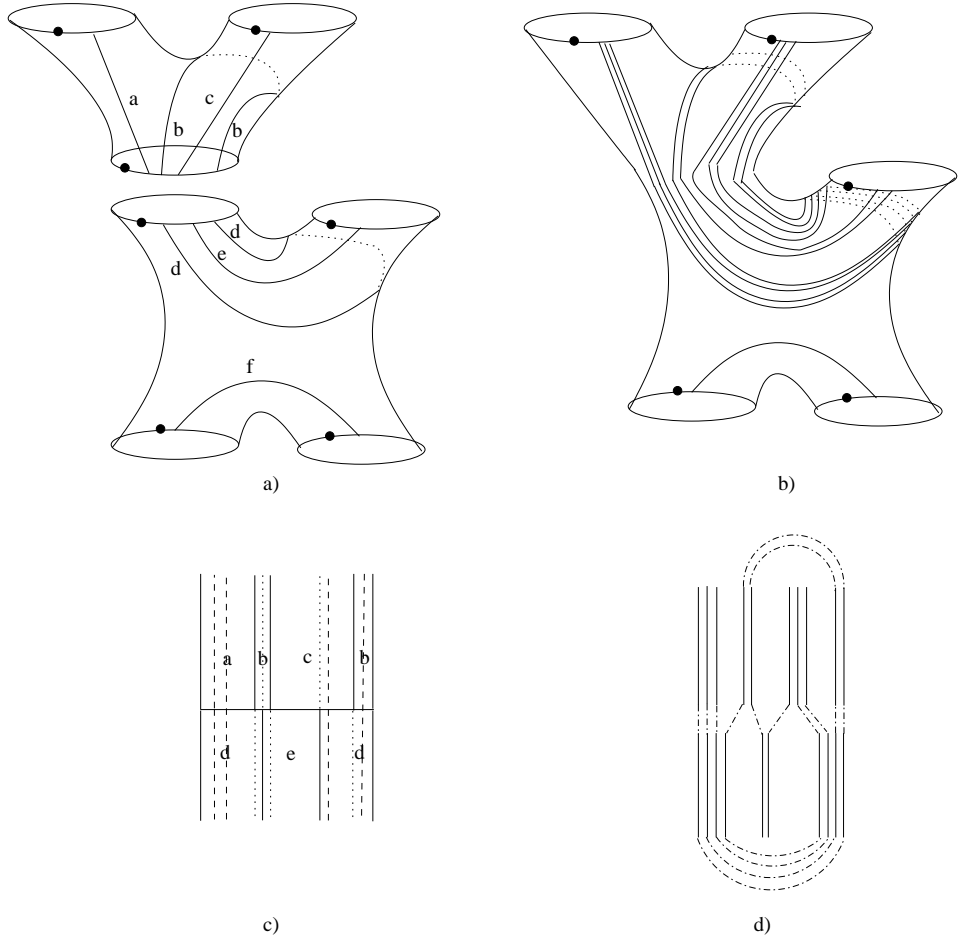}
\caption{\label{gluingex1} Example of gluing
the top arc family to the bottom arc family.
a) The arc graphs which are to be glued
assuming the relative weights a,b,c,d and e as indicated by the
solid lines in c). b) The result of the gluing (the weights are
according to c). c) The combinatorics of cutting the bands. The
solid lines are the original boundaries, the dotted lines are the
first cuts and the dashed lines represent the recursive cuts. d) The
combinatorics of splitting, and joining flags.}
\end{figure}

\subsubsection{Cutting: ``co--operad structure''}
We will often be interested in the dual structure to gluing, that of
cutting. In order to cut a surface into two components such that
their operadic composition is the original surface we have to
specify the following data: a separating curve $c$ and a point $p$
on $c$. The point $p$ can actually be arbitrary. In order to cut, we
simply cut along $c$ and make the images of $p$ the marked points on
the two now boundaries. Notice that when we glue, $p$ and $c$ just
disappear.

\subsubsection{Subspaces}
\label{subspaces} We would like to recall and introduce the
following notation for subspaces.\\

\begin{tabular}{l|l}
Subspace&Condition\\
\hline\\[-2mm]

$\Arc^s_{\# \; g}(n)\subset \Arc_{g}^s(n)$& complementary regions
are polygons \\&or once punctured
polygons.\\

$\GTree(n) \subset \Arc(n)$ & $s=0$ and all arcs run only
from boundary $0$ to some\\& boundary $i\neq 0$.\\
 $\CGtree(n)\subset \GTree(n)$ &  the {\em cyclic} order
of  the arcs at the boundary $0$\\& is anti-compatible with the
linear order at each\\& other  boundary. I.e.\ if $\prec_i$ is the linear
order at $i$ then\\&  $e\prec_i e'$ is equivalent to
$e'\prec_0
e$.\\

 $\LGtree(n)\subset \GTree(n)$ &  the {\em linear} order
of  the arcs at the boundary $0$\\& is anti-compatible with the
linear order at each\\& other boundary.\\

$\Corol$&exactly one arc for each boundary $i\neq 0$\\& which runs
to boundary $0$.\\
\end{tabular}\\[2mm]

We will use the subscript $cp$ to signify $g=s=0$ and use the notation
$\Tree:=\GTree_{cp}$ and $\Lintree:=\LGtree_{cp}$. Notice that $\CGtree_{cp}=\Tree$, since the condition is guaranteed by the condition $g=s=0$.

\begin{nota}
\label{genusnota}
For a collection of subspaces $\calS(n)$ as above
we will write $\calS_g(n)$ to indicate
that the genus and the number of boundary components are fixed, of course
$\calS(n)=\amalg_g \calS_g(n)$.
The symbol $\calS$ as a space will stand for $\amalg_{g,n} \calS_g(n)$
and as an operad for the collection $\{ \calS(n)\}$.
\end{nota}

\subsubsection{De-projectivized arcs}
In order to get {\em isomorphisms} with cacti \cite{del,vor} one
needs to include a factor of $\Rp$ in these operads. The process was
called de--projectivizing in \cite{KLP}. Skipping this step one
still obtains equivalences.

\label{darcs}
\begin{df}
 Let $\Darc_{g,r}^s := \Arc_{g,r}^s\times \mathbb{R}_{>0}$.
\end{df}
 The elements of $\Darc$ are graphs on surfaces with a
metric, i.e.\ a function  $w:E_{\Gamma}\rightarrow
\mathbb{R}_{>0}$. Furthermore $\Darc$ is a cyclic operad
equivalent to $\Arc$ \cite{KLP}. The operad structure on $\Darc$
is given as follows. Let $\a,\a'$ be elements of $\Darc$, if the
total weight at the boundary $i$ of $\a$ is $\lambda$ and the
total weight at the boundary $0$ of $\a'$ is $\mu$, then first
scale the metric $w$ of $\a$ to $\mu w$ and likewise scale the
metric $w'$ of $\a'$ to $\lambda w'$ and afterwards glue as above.

\begin{nota}
Any collection of subspaces $\mathcal{S}$ of $\Arc$ defines a collection
of subspaces
$\mathcal{DS}:=\mathcal{S}\times \mathbb{R}_{>0}$ of $\Darc$.
\end{nota}

\begin{prop} For any suboperad $\mathcal S$ of $\Arc$ there
are  {\em isomorphisms} of operads
$\mathcal{DS}/\mathbb{R}_{>0}\simeq \mathcal{S}$ where
$\mathbb{R}_{>0}$ acts by scaling on the right factor
$\mathbb{R}_{>0}$ of $\Darc$. And these isomorphisms induce
equivalences of operads: ${\mathcal DS}\sim {\mathcal S}$.
\end{prop}
\qed

\begin{thm}\cite{KLP,cact,del,hoch1}
$\GTree,\LGtree$ and $\Corol$ as well as their restrictions to
$g=s=0$ are suboperads (not cyclic) of the cyclic operad $\Arc$. The
same holds for their versions in $\Darc$ defined above. The spaces
 $\Arc_{\#,g}^0(n)$
form a rational suboperad and
$\mathcal{D}\Arc_{\#,g}^0(n)$ is a rational suboperad of $\Darc$.
(Here rational means that the compositions only need to be defined
on a dense open set.) Furthermore the following relations hold,
where the first line only holds on the level of rational operads.

\begin{tabular}{l|ll|ll}
Suboperad& \multicolumn{2}{l|}{isomorphic operad}&\multicolumn{2}{l}{equivalent operad}\\
\hline&&\\[-2mm]
 $\mathcal{D}\Arc_{\#,g}^0(n)$&$M^{1^{n+1}}_{g,n+1}$&\cite{hoch1}& \\
$\mathcal{D}\Tree$&$\Cacti$ &\cite{KLP}&$fD_2$ & \cite{cact}\\
$\mathcal{D}\Lintree$&$\Cact$
\settowidth{\ghost}{$i$}\makebox[1\ghost]{}
 &\cite{KLP}&\settowidth{\ghost}{$f$}\makebox[1\ghost]{}$D_2$ &\cite{cact}\\
 $\mathcal{D}\Corol_{cp}$&$\SCC$&\cite{cact}&\settowidth{\ghost}{$f$}\makebox[1\ghost]{}$A_{\infty}$&\cite{cact}\\
\end{tabular}

Additionally $\CGtree$ is also a suboperad.
\end{thm}

\begin{proof}
The only statement not contained in the references is the one about
$\CGtree$. This  follows, however, in a  straightforward fashion
from the gluing procedure. Alternatively one can use Proposition
\ref{biprop} below.
\end{proof}

\begin{rmk}
Although the first line only deals with rational operads,
 it induces an isomorphism of true operads on the chain
level \cite{hoch1}. Here $M^{1^{n+1}}_{g,n+1}$ is the moduli space
of genus $g$ curves with $n$ marked points and a tangent vector at
each of these marked points. The operads in the second column are as
follows: $\Cacti$ is the operad of cacti introduced in \cite{vor},
$\Cact$ is the operad of spineless cacti \cite{cact} of and $\SCC$
is the suboperad of spineless cacti with only one vertex.
The operads in the third column are the familiar ones, that is
 $D_2$ is the $E_2$ operad of little discs, $A_{\infty}$ is the
$E_1$ operad of little intervals and $fD_2$ is the framed little discs operad.
The inclusion of $\GTree_{cp}\subset \Arcn$ thus gives an  $BV_{\infty}$ (BV up to homotopy) structure to a cell model of moduli which includes an
$A_{\infty}$ structure.

\end{rmk}

We will deal with $\Gtree$, $\LGtree$ and $\CGtree$ in the
following.

\subsection{Extended  gluing}
\label{extendedpar} The gluing procedure above was defined when gluing
together two boundaries which have the same width of the foliations.
The space $\Arc$ was chosen to guarantee that the boundaries are hit
and hence can be scaled to agree.
The  extension of the gluing we wish to make is to sometimes allow
gluing on a boundary with no incident arcs. In this case we glue the
surface and {\em delete} all the arcs incident to the boundary we
are gluing onto.

There will also be a gluing, where we will remember the deleted
arcs. This is described in detail in \S\ref{thickendpar}. In this gluing
we allow gaps in the foliation of a given width at the boundary $0$.

As an alternative to scaling the whole surface as in $\Darc$ we will
consider scaling only those arcs incident to the boundaries to be
glued. There are three types of scalings which provide glueable
foliations. Homogeneously scaling the arcs (1) at boundary $i$ or (2) at the
boundary $0$ of the other surface or (3) symmetrically scaling. We will
use the version (1) where we scale the arcs of the boundary $i$.

\section{The operad $\Sttree$}

\begin{assump} From here on out, we will assume that there are no punctures.
 Consequently we will set $s=0$ and {\em drop the superscript $0$}
 from the terminology of \cite{KLP}, e.g.\ we write $\Arc_g(n)$ for $\Arc_g^0(n)$.
\end{assump}

\subsection{Technical setup}

\subsubsection{Euler characteristic and quasi--filling arc graphs}
\begin{df}
We define the Euler characteristic of an element $\alpha \in
A_{g,r}^s$ to be $\chi(\alpha)=|\comp(\a)| - |E_{\Gamma(\alpha)}|$.
\end{df}

\begin{prop}\cite{hoch1}
\label{Euler}
 The following inequality holds
\begin{equation}
\chi(\alpha)\geq \chi(F(\alpha))
\end{equation}
and the equality holds if and only if the complementary regions
are polygons.
\end{prop}

The difference $\chi(\a)-\chi(F(\a))$ measures the defect of the surface.

\begin{df}

We set $\edefect(\a)=\sum_{R\in \comp(\g(\a))} (\chi(R)-1)$ and call
it the Euler defect.
If the Euler defect is  $0$  we call the elements quasi--filling.
Otherwise the element is called unstable.
\end{df}

\begin{ex} The elements
$\T_a$ have Euler defect $0$ and for the graphs in the Figure \ref{unstable}:
$\H_a$ has Euler defect $-1$ and $\G$ has Euler defect $-2$.
\end{ex}

\begin{lem}
\label{eulerlem}
The Euler defect defines an operadic filtration on $\Arc$ by
$\Arc^{(i)}$ where these are the elements of at most defect $-i$.
\end{lem}

\begin{proof}
It is clear that the defect may only drop, since $\chi-1$ is
additive under gluing the complementary regions, if there is no
self--gluing; and if there is self--gluing then the Euler
characteristic decreases. For a careful analysis of all the
combinatorics that can occur see \cite{hoch1}.
\end{proof}

\subsubsection{Twisting at the boundary}
\begin{df}
We define the twisting at the boundary $i\neq 0$ of $\a$ by an angle
$a$ to be the composition $\a\circ_i \T_a$ and at the boundary $0$
we define the twist to be $T_{1-a}\circ_1\a$.
\end{df}

Notice $T_a\circ T_b=T_{a+b}$ here we calculate in $\R/\Z$. The
effect of a twist is to move the boundary point by the angle $a$
measured in units of $2\pi$.

\begin{df}
An arc graph is called {\em twisted at the boundary $i$} if  the first
and last edges at that boundary
 become homotopic, if one allows the endpoint on the boundary
$i$ to vary considering the marked point of the boundary $i$ as part
of the boundary.

An arc graph is called {\em untwisted} if it is not twisted at any
boundary. It is called {\em possibly twisted at $0$} if it is
untwisted at all boundaries $i$ with $i\neq0$.

An element of $\A$ is called twisted or untwisted at a boundary if
the underlying graph is. And likewise possibly twisted at $0$ if its
arc graph is.
\end{df}

\begin{lem}
An element $\a\in \Arc$ twisted at the boundary $i$ can be decomposed as
$\a'\circ_i\t$ if $i\neq 0$ or $\t\circ_1\a'$
for some $\t\in\Arc_0(1)$ and $\a'$ not twisted at the boundary $i$.
\end{lem}

\begin{proof}
If a boundary is twisted, then it becomes untwisted by moving the
boundary point through one of the two parallel bands. This
corresponds to a composition with $\T_a$ for some $a$. Since
$T_{a-1}\circ_1 T_{a}=T_0$, if we assume that $i\neq 0$ we see that
$\a=\a'\circ_i T_{a-1}$ with $\a'=\a\circ_i T_a$ and analogously for
$i=0$.
\end{proof}

Let $\Gamma$ as usual denote the full operadic composition as opposed
to the pseudo operadic compositions $\circ_i$.

\begin{cor}
\label{twistedcor}
Any $\a \in \Arc(n)$ can be written as $\Gamma(\Gamma(\t_0,\a'),\t_1,\dots,\t_n)$
with $\t_i\in \Arc_0(1)$ and $\a'$ untwisted.
\end{cor}

\begin{rmk}
Note that this decomposition is not canonical in general.
\end{rmk}

\begin{lem}
\label{cuttinglem}
When cutting an element into two elements,
 we can always choose the
point $p$ on the cutting curve in such a manner that one of the new
boundaries is  untwisted.
\end{lem}

\begin{proof}
For this we first consolidate all bands
that become parallel after cutting on one of the two surfaces.
Now we choose the point $p$ not to lie inside any of these consolidated bands.
\end{proof}

\subsection{The structure of $\Gtree$}
In this section, we show that elements in $\Gtree$ have a standard
decomposition in terms of twists, an unstable element in $\Gtree(1)$
---which can be decomposed into canonical elements from
$\GTree_1(1)$--- and an untwisted quasi--filling element; here we used
Notation \ref{genusnota}.

\subsubsection{Twisting in $\Gtree$}
\begin{lem}
If $\a,\b\in \Gtree$  are both untwisted then for every possible
$i$: $\a\circ_i\b$ is untwisted.

If $\a,\b\in \Gtree$ are both possibly twisted at zero,
then for every possible $i$: $\a\circ_i\b$ is possibly twisted at
$0$.
\end{lem}

\begin{proof}
The first statement is immediate. For the second statement, we have
to use the fact that all arcs run to $0$. Thus after gluing, if two
arcs would become parallel after allowing the endpoints to vary
across the marked point on a boundary different from $0$, then they
would have to be parallel starting at the boundary $0$ up to the
separating curve which was the glued boundary and furthermore the
marked point on that curve would have had to lie between them. Hence
the two arcs in question have to be continued by parallel arcs,
contrary to the assumption.
\end{proof}

\begin{cor}
The subspaces of  untwisted elements and those of possibly twisted
at $0$ are suboperads. The former will be given a superscript $'$
and the latter a superscript $0$.
\end{cor}

Particular examples are $\CGtree^0=\LGtree$; and the suboperads
$\LGtree'$ and $\GTree^0$ and $\GTree'$.

\begin{lem}
\label{twistedlem} Any element $\a\in\CGtree(n)$ can be canonically
written as $\a\in\Gamma(\a',\t_1,\dots,\t_n)$ with $\t_i\in
\Arc_0(1)$ and $\a'\in\LGtree(n)$.
\end{lem}

\begin{proof}
Just like in \cite{cact} the main point is that the space
$\CGtree(n)$ is a trivial $(S^1)^{\times n}$ bundle over
$\CGtree^0(n)=\LGtree$. The fact that forgetting the marked points
on the boundaries different from zero is an $(S^1)^{\times n}$
bundle is clear. The section is given as follows: all the bands hit
$0$ and the cyclic orders are compatible. This means that going
around the boundary zero for each boundary there is a first band
that appears. The first leaf of this band defines a canonical point
on the $i$--th boundary. Now the marked point on this boundary is
then determined by the distance (using the partial measure on the
foliation) from this point. Since this map depends continuously on
the marked point at $0$ and the other marked points, it  gives
global co--ordinates and a global trivialization. In particular, the
element $\a'$ above is the element where the canonical points are
marked and the $\t_i$ are the elements $T_{a_i}$ which twist by the
distance.
\end{proof}

\begin{rmk}

In general on $\Gtree$ the section constructed above is actually only piecewise
linear and may become discontinuous as soon as the genus is bigger or equal to
one.  The compatibility of the cyclic orders was key above. If a
braiding occurs, the result ceases to be true.
\end{rmk}

This allows us to decode the structure of $\CGtree$ as the
generalization of $\Cacti$.
Recall (see e.g.~\cite{cact})
that for any monoid $M$ there is an operad $\mathcal M$ given
by taking ${\mathcal M}(n):=M^{\times n}$ with the permutation action
and the compositions given by using the diagonal embedding
and the multiplication of the monoid.

\begin{prop}\cite{cact}
 $\Tree$ is the suboperad of $\Arc$ generated by $\Lintree$
and $\Arc_0(1)$. Moreover it is a bi--crossed product of  $\Lintree$
with the operad built on the monoid $\Arc_0(1)\simeq S^1$.
\end{prop}

\begin{prop}
\label{biprop}
$\CGtree$ is the suboperad of $\Arc$ generated by $\Arc_0(1)$ and $\LGtree$.
 Moreover it is a bi--crossed product
of  $\LGtree$ with the operad built on
the monoid $\Arc_0(1)\simeq S^1$.
\end{prop}

For the definition of bi--crossed products see \cite{cact}.
\begin{proof}
The first part follows from Corollary \ref{twistedcor}. In view of Lemma \ref{twistedlem} and
its proof, the proof of
the bi--crossed product part for $\Gtree$
is  analogous to the argument given in \cite{cact}
for $\Tree$.
\end{proof}

\subsubsection{Classifying elements in $\Gtree$}

\begin{lem}
\label{decomplemma1} Any unstable element $\a\in \Gtree(n)$ can be written as
$\a_1\circ_1\a'$ with $\a_1\in \Gtree(1)$ unstable  and $\a'$
quasi--filling.
\end{lem}

\begin{proof}
Since all the complementary regions border the boundary $0$ we can
decompose $\a$ as $\a_1\circ_1 \a'$ with $\a'$ quasi--filling and $\a_1$
unstable by ``sliding down'' the defects and cutting with a
separating curve.
\end{proof}

An example of this procedure is given in Figure \ref{unstableex}.

\begin{figure}
\epsfxsize = .8\textwidth \epsfbox{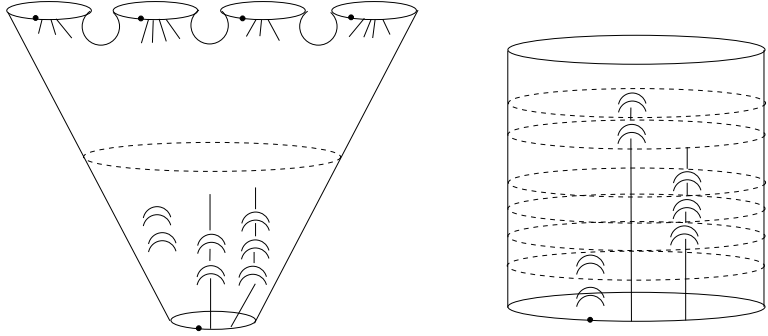}
\caption{\label{unstableex} Decomposing an unstable arc graph into
standard form. First ``slide down'' the defects, viz.\ decompose
$\a$ into $\a'$, the top part, and $\a_1$ the bottom part. Second
further decompose $\a_1$ by cutting so that there is one defect $H$
or $G$ per ``ring''.}
\end{figure}

\begin{lem}
\label{finerlem}
In the above decomposition, we can furthermore decompose $\a'$ as
$\a''\circ \a_0$ where $\a''\in \Gtree(1)$ is quasi--filling
and $\a_0$ is in $\Gtree_0(n)$ with either $\a_0$ not twisted at $0$
or $\a''$ not twisted at $1$.
\end{lem}

\begin{proof}
As above we can choose a cutting curve which separates the surfaces
as stipulated. The additional condition about being untwisted
follows from Lemma
 \ref{cuttinglem}
while the fact that
both $\a_0$ and $\a''$
 have to be quasi--filling follows from Lemma \ref{eulerlem}.
\end{proof}

\begin{lem}
\label{structurelemma} If $\a$ is an untwisted unstable element of
$\Gtree_1(1)$ then up to twists $\a$ is either of the form $\G$ or
$\H_a$ as depicted in Figure \ref{unstable}.
\end{lem}

\begin{proof}
The proof is a straightforward consideration in geometric topology.
Since we are working up to a twist, we will omit the marked points
in the consideration. If there is only one arc then up to the action
of PMC the element is $\G$. Say we have several arcs. We cut along
the first arc which after PMC action we can assume to be as in
Figure \ref{S21}. The resulting  surface will have one boundary
component and genus $1$; see Figure \ref{S21}. If there is a second
arc, we can put it into the position as in Figure \ref{S212}. After
cutting along this second arc, the situation is as in the last part
of Figure \ref{S212}. But in this figure any arc running from a
piece of the boundary marked by $0$ to a piece of the boundary
marked by $1$ will cut the surface into a polygon.
\end{proof}

\begin{figure}
\epsfxsize = .8\textwidth \epsfbox{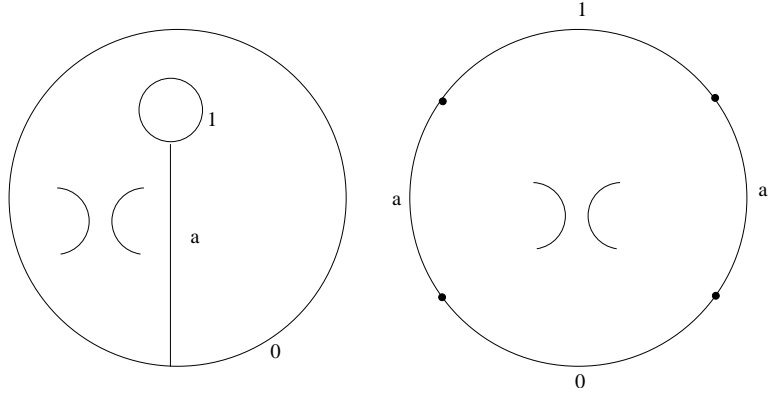}
\caption{\label{S21} An unstable graph in $\Gtree_1(1)$ with one arc}
\end{figure}
\begin{figure}
\epsfxsize = .8\textwidth \epsfbox{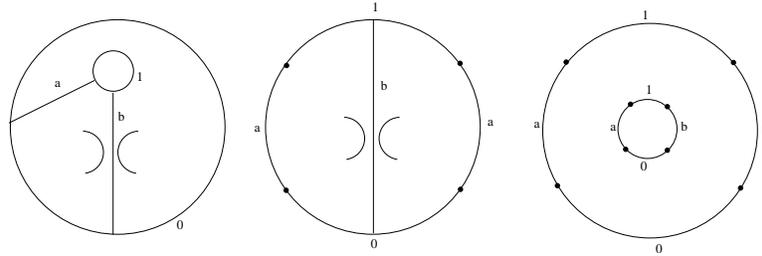}
\caption{\label{S212} An unstable graph  in $\Gtree_1(1)$ with two arcs}
\end{figure}

\begin{figure}
\epsfxsize = .5\textwidth \epsfbox{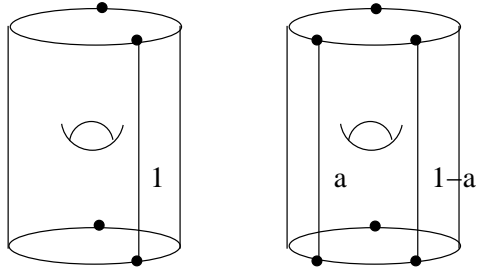}
\caption{\label{unstable} The two basic unstable arc graphs $\G$ and $\H$.}
\end{figure}

We will call  the sum of genera of the complementary regions the
genus defect and the sum of the number of boundaries minus one of
the complementary regions the boundary defect  (see the Appendix for
more details).

\begin{lem}
Any element $\a$ in $\Gtree(1)$ can be written as
\begin{equation}
\label{decompeq} \a=T_{a_0}\circ_1 \G\circ_1 T_{a_1} \circ_1\dots \circ_1
\G \circ_1 T_{a_k}\circ_1 \H_{b_1}\circ_1 T_{a_{k+1}} \dots\circ_1
\H_{b_l}\circ_1 T_{a_{k+l}} \circ_1 \a'
\end{equation}
with $\a'$ quasi--filling and not twisted at $0$. Furthermore $k$ is
the  sum of the genus defects of the complementary regions and $l$ is the sum of
boundary defects.
\end{lem}

\begin{rmk}
There are no free boundary defects in $\Gtree$, since all boundaries get hit.
\end{rmk}

\begin{proof}
After separating off the quasi--filling part, by the Lemma \ref{decomplemma1},
we can use separating curves to cut $\a$ so that there is at most one
handle in each piece. We can furthermore arrange the handles with no
curve passing through to be cut first.
\end{proof}

An example of this procedure is given in Figure \ref{unstableex}.

\begin{prop}
\label{structureprop} We have the following identities in $\Gtree$:
$\G\circ_1H_a=H_a\circ_1 \G$ and $\H_a\circ_1 H_b=
H_{b-1}\circ_1H_{a-1}$ and furthermore for any $\a\in \GTree:$
$\a\circ_i \G=\G\circ_1\a$ and there is some $b$ such that
$\a\circ_i \H_a=\H_b\circ_1\a$.

If $\a=\a_1\circ_1\a'$ and $\b=\b_1\circ \b'$ as in Lemma
\ref{decomplemma1} then $\a\circ_i \b = \g_1 \circ_1 \g'$ in the
same notation with $\g'= \a'\circ_i T_b \circ_i \b'$ where $b$ is
the sum of all the twists in $\b_1$.
\end{prop}

\begin{proof}
The first part is straightforward. For the relations for $G$ and
$\H$ we notice that we can ``pull--down'' the handle and cut it off
just like before. Then the last part follows since the intermediate
twists will all add up.
\end{proof}

\subsection{Stabilizing at $0$}
Notice that the compositions  $\a\mapsto T_{-a} \circ_1 \H_b \circ_1
T_a \circ_1 \a$ and $\a\mapsto T_{-a} \circ_1 \G \circ_1 T_a\circ_1\a$
give maps: $\sth^g(a,b):\Gtree_g(n)\to\Gtree_{g+1}(n)$ and
$\stg^g(a):\Gtree_g(n)\to\Gtree_{g+1}(n)$.

\begin{df}
We define $\Sttree(n):=\colim_{\calS} \Gtree(n)$ where the co\-limit
is taken over the system of maps $\calS$ generated by $\stg^g(a)$
and $\sth^g(b,c)$ with $a,b\in [0,1)$ and $c\in (0,1)$. We will
denote the image of a subspace by the prefix $\St$, e.g.\
$\St\LGtree$.
\end{df}

We could of course also use that $G=H_0=H_1$, but the above is maybe
more natural. An example of stabilization is given in Figure
\ref{stabilize}.

\begin{figure}
\epsfxsize = .5\textwidth \epsfbox{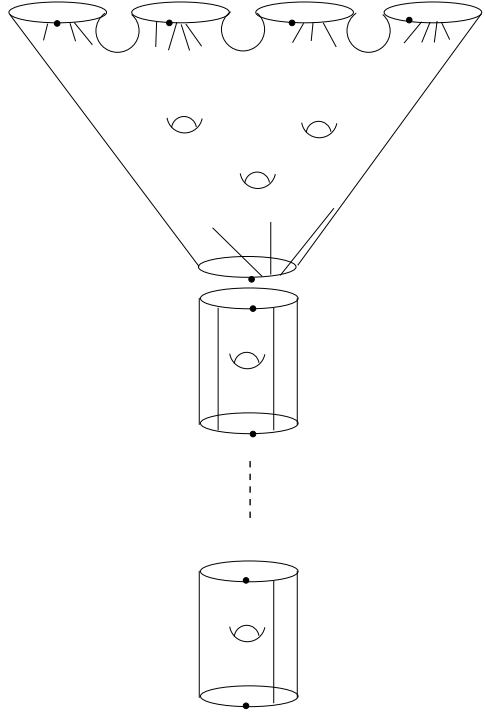}
\caption{\label{stabilize} Stabilizing.}
\end{figure}

\begin{prop}
\label{stimprop} The image of an element $\a\in \Gtree$ is given as
follows. Let $\a$ be decomposed as in equation (\ref{decompeq})
then, $[\a]=[T_b\circ_1 \a']\in \Sttree$ with $b$ the sum of all the
twists.
\end{prop}

\begin{proof}
Decompose equation (\ref{decompeq})
\begin{multline}
(T_{a_0} \circ_1 \G  \circ_1 T_{-a_0}) \circ_1
(T_{a_0+a_1} \circ_1\G \circ_1T_{-a_0-a_1})\circ_1 \dots \circ_1(T_{a_0+\dots+ a_{k-1}} \circ_1 G\circ_1
T_{-(a_0+\dots+ a_{k-1})} ) \circ_1\\
(T_{a_0+\dots+ a_{k}}\circ_1  H_{b_1}\circ_1T_{a_0+\dots +a_{k}})
\circ_1 \dots
\circ_1 (T_{\sum_0^{k+l-1}{a_i}} H_{b_l} \circ_1
T_{-\sum_0^{k+l-1}{a_i}})\circ_1 T_{\sum_0^{k+l}{a_i}}\circ_1\a'\\
=\stg(\bar a_0)\circ \dots \circ \stg(\bar a_{k-1})\circ\sth(\bar a_k,b_1)\circ\sth(\bar a_{k+l-1},b_l)(T_{\bar a}\circ_1 \a')
\end{multline}
where $\bar a_j:=\sum_{i=0}^j a_i$.
\end{proof}

\begin{cor}
\label{quasicor} As spaces $\St\LGtree(n)=\LGtree_{\#}(n)$ that is
the quasi--filling elements of $\LGtree$.
\end{cor}

\begin{thm}
The operad structure of $\Gtree$ descends to $\Sttree$. Moreover
$\St\LGtree$ and $\St\LGtree'$ are suboperads.
\end{thm}
\begin{proof}
The fact that the operad structure descends is a direct consequence
of Proposition \ref{stimprop} and Proposition \ref{structureprop}.
Since the stabilization adds a net twist of zero, the claims for
the suboperads hold true.
\end{proof}

\subsection{Degeneracies and thickening $\St\LGtree$}
\subsubsection{Preoperads and weak unital operads}
Recall that a preoperad is given by a collection $\calO(n),n>0$
together with $\Sn$ actions and degeneracy maps $s_i:\calO(n)\to
\calO(n-1)$, see (e.g.~\cite{MSS, Berger} for details) which are
$\Sn$ equivariant and satisfy the usual relations.

According to the language of \cite{may} an operad $\calO=\{\calO(n)\},n\geq 0$
is unital if $\calO(0)=*$ that is it is a point.
Any unital operad gives a preoperad by forgetting all the structure maps,
except the composition with $\calO(0)$ and the identity ${\bf 1}$ in $\calO(1)$:
$s_i(a):=\Gamma(a;{\bf 1},\dots,{\bf 1},*,{\bf 1},\dots,{\bf 1})$ where
$*$ is in the $i$--th position.

We would also like to consider the new notion of a {\em weak unital operad}
which is given by an operad  $\calO=\{\calO(n)\},n> 0$ together
with a preoperad structure on the collection $\calO$. Any unital operad
yields a weak unital operad by forgetting $\calO(0)$ but retaining
the induced degeneracy maps. Finally we call a weak unital topological
operad a quasi--unital operad if $\calO(n),n>0$ form an operad $\calO(0)=*$
and the $\calO(n),n\geq 0$ form a quasi--operad (viz.~homotopy associative)
and defines
a preoperad structure.
The degeneracies only need to  commute with the other operadic compositions
up to homotopy.

\begin{nota}
For a preoperad,
we define $\phi^*_{ij}:\O(n)\to \O(2)$,$1\leq i<j\leq n$  by using
the degeneracy maps in all entries
{\em everywhere except} at $i$ and $j$.

For an operad  $\O$ with a fixed element $*\in \O(0)$
 $\phi^*_{ij}$ is the map that
glues in $*$ {\em everywhere except} at $i$ and $j$, viz $a\mapsto
\Gamma(a;*,\dots,*,{\bf 1},\dots,*,{\bf 1},*\dots,*)$, where ${\bf 1}$
is in the $i$--th and $j$--th position.

\end{nota}

\subsubsection{Adding degeneracies}

The operad $\Gtree$ does not have
a $0$--th space. We can add a $0$--th space $\Gtree(0)$
consisting of all surfaces $F_{g,1}$, that is
genus $g$ with one boundary component
that has a marked point and the empty foliation.
The operadic compositions are given by the extended gluing,
which erases arcs.

There is a special element $*$ which is  the disc $D$ with a marked
point on the boundary and without any arcs.  For any arc family $\a$
we define $s_i(\a)$ to be the arc family resulting from gluing in
$D$ into the $i$--th boundary using the extended gluing of \S
\ref{extendedpar}.

\begin{prop}
Adding $\Gtree(0)$ gives $\Gtree$ the structure of
a quasi--operad. Using $*$ to define
degeneracies $s_i$ gives
$\Gtree(n),n>0$ the structure of a preoperad. This structure descends to the
stabilization, where $\St\Gtree(0)$ is a point which is the image of $D$.
The collection $\St\Gtree(n),n\geq0$ is a quasi--unital operad.
\end{prop}

\begin{proof}
First we notice that the spaces are stable under the extended
gluing. The effect of gluing in a surface with one boundary and an
empty foliation is that the boundary is filled in by the surface and
all the arcs running to this boundary are deleted. First we notice
that indeed this decreases the boundaries by one and secondly the
result is still in $\Gtree$ if $n\geq 2$ as all arcs still run only
from the boundaries $i\neq 0$ to $0$. In case that $n=1$ after
gluing in the surface we erase all arcs and obtain a surface with an
empty foliation.

Secondly although the gluing using deletion is not strictly associative,
it is homotopy associative. The homotopy which is
tedious to write out is given by increasing and decreasing
the weights according to the erased weights. Another proof of this fact
comes from the thickening construction below.

The
associativity and ${\mathbb S}_n$ equivariance are clear. The
fact that stabilization goes over well is straightforward.

Finally notice that we can write any of $F_{g,1}$ with the empty foliation
as the composition of
$G\circ_1G\circ_1\dots \circ_1 G \circ_1D$ with $g$ factors of $G$.
This proves the last statement.
\end{proof}

\subsection{Thickening}
Although $\St\LGtree$ can be made into a quasi--unital operad by
simply adding a point which is the image of the disc as a $0$--th
component, the extended gluing will fail to be associative on the
nose, however, and will only be associative up to homotopy. This
would of course be enough for the homology level and is even enough
for a cellular chain model (see \S \ref{chainpar}), but in order to
use the results of \cite{Berger} and the recognition principle of
\cite{may2}, we will have to have a {\em bona fide} unital operad.

In order to achieve this we will thicken our construction just enough
to keep track of the homotopies involved.

\subsubsection{Thickening the operad}
\label{thickendpar} As spaces we define $\Fat\calD\Gtree(n)$ for
$n>0$ to be given by pairs $(\a,\gap)$ where $\a$ is a generalized
weighted arc graph on a surface $F_{g,n+1}$ with marked points on
the boundary and $\gap$ is a gap labelling function. In particular
the graphs we consider are PMC orbits of exhaustive graphs on
surfaces whose edges all run from the boundaries $i\neq 0$ to $0$,
where we now allow the edges to be parallel,  and $\gap$ is a map
$\gap:*\amalg E_{\Gamma}\to \R_{\geq 0}$, such that if $e$ and $e'$
with $e\prec_0 e'$ are parallel then $\gap(e)>0$. We think of the
value of $\gap$ on an edge as the width of gap after this edge and
the value  $\gap(*)$ as the gap before the first edge.

We let $|\gap|:=\sum_{e\in E_{\Gamma}}gap(e)+\gap(*)$. Scaling
$(\a,\gap)$ by $\lambda\in \R_{\geq 0}$ means that we simultaneously
scale all weights of $\a$ and scale $\gap$ to $\lambda\gap$ where
$(\lambda\gap)(x)=\lambda \cdot \gap(x)$. We call the value
$\gap(x)$ the width of the gap. The width may be zero. The total
weight at zero  $|\gap|+\sum_{e\in \Gamma}w(e)$ will be positive.

We define $\Fat\calD\Gtree(0)$ be the set of pairs $(F',\gap)$ where
$F'=F_{g,1}$ is a surface with marked points on the boundary
considered to have an empty foliation and $\gap:*\to\Rp$ is
arbitrary. Notice that this makes the total weight at $0$ positive.

\subsubsection{Thickened gluing}
The composition $(\a,\gap)\circ_i (\a'\,gap')$ is defined to be the
pair $(\a'',\gap'')$ obtained as follows: first glue the surfaces as
previously; secondly glue the foliations and gaps in the following
perturbed way. As before and in \cite{KLP} we fix a measure on the
surface to turn edges with weights into bands of a foliation. As in
{\em loc.~cit.} the construction does not depend on this choice.

\begin{enumerate}
\item Let $w_1$ be the sum $|\gap|$
and the weights at $0$ of $\a'$. Let $w_2$ be the sum
of the weights at $i$ of $\a$.
\item Scale $(\a,\gap)$ by $w_1$ and $(\a',\gap')$ by $w_2$.
\item Glue the foliations along an interval of
width $w_1w_2$ as follows. Arrange the foliations on the interval so
that the ends of the scaled foliation at the boundary fill out the
interval on one side and on the other side arrange the bands in the
following way. The initial point of the interval corresponds to the
marked points on the boundary. First, leave a gap of width
$w_1\gap(*)$ then attach the first band corresponding to $e_1$ with
width $w_1w(e_1)$, then again leave a gap this time of width
$w_1\gap(e_1)$ and so on. Now, a) fuse leaves which share the same
endpoint and b) if leaves end on a gap erase the leaf, but mind the
width of the gap and add its weight to the gap(s) at the boundary
$0$ to which the band of erased leaves are adjacent. This may result
in the creation of new non--zero gaps or the consolidation of
several gaps.
\item Remove any closed leaves.
\end{enumerate}

\begin{figure}
\epsfxsize = \textwidth \epsfbox{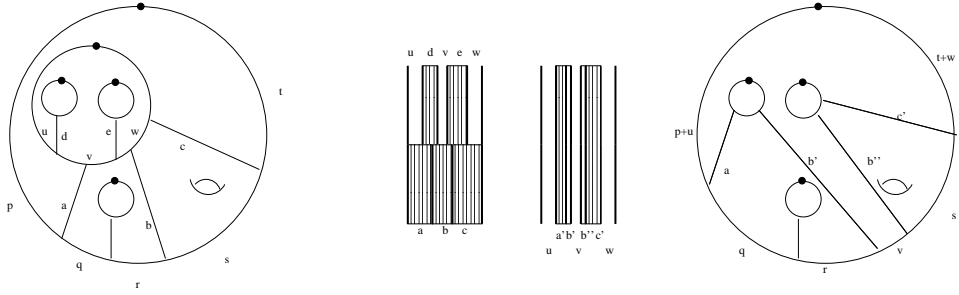} \caption{\label{gaps}
Gluing with gaps. The weights at the glued boundaries are $a,b,c$ and $d,e$ while
the weight of the gaps at the glued boundary $0$ are $u,v,w$ and the width of the gaps
at the boundary $0$ that is not glued are $p,q,r,s,t$.
The new weights satisfy the equations:  $a=a'+u,b=b'+v+b'', c=c'+w, d=a'+b',e=b''+c'$}
\end{figure}

Just like for $\calD\Arc$ there is a transitive $\Rp$ action given
by  scaling $\Rp\times \Fat\calD\Gtree\to \Gtree$ where
$\lambda(\a,\gap):=(\lambda\a,\lambda \gap)$.

We set $\Fat\Gtree(n):=\Fat\calD\Gtree(n)/\Rp$.

\begin{prop} The spaces $\Fat\calD\Gtree(n),n\geq 0$
form an operad using the gluings described
above and the permutation action on the boundary labels.
This operad structure descends
to the collection $\Fat\Gtree(n),n\geq 0$.
\end{prop}

\begin{proof}
This is a straightforward but tedious check. The basic reasoning is
that instead of erasing the leaves of the foliation, we can leave
them ending on the separating curve that is the image of the boundaries under
the gluing. Gluing in this way
is associative. Now we  can erase the respective
leaves after all the gluings are done and this coincides with the previously
defined gluing.
\end{proof}

\begin{prop}
\label{retractprop}
There are operadic inclusions\\
$\{\calD\Gtree(n),n>0\}\hookrightarrow \Fat\calD\Gtree$
and $\{\Gtree(n),n>0\}\hookrightarrow \Fat\Gtree$.

Furthermore $\Fat\calD\Gtree(n), n\geq 0$ retracts onto
$\calD\Gtree(n)$ and\\ $\Fat\Gtree(n),n\geq0$ onto $\Gtree(n)$.
\end{prop}

\begin{proof}
The operadic inclusion is given by $\a\mapsto (\a,0)$ where $0$ is
the constant map with value $0$. For $n>0$ the retraction is given
by scaling $\gap$ to $0$ and consolidating bands corresponding to
parallel edges, by adding the weights of parallel edges and keeping
only one edge per set of parallel edges. For $n=0$ we just contract
$\Rp$ to the point $1$.
\end{proof}

\begin{cor}
\label{equivcor}
The operads $\{\Fat\calD\Gtree(n), n>0\}$ and $\{\calD\Gtree(n),n>0\}$
are equivalent.
\end{cor}

\subsubsection{Stabilizing $\Fat\Gtree$}
Since we have established the inclusion, we have the system of
maps $\calS$ generated by $\stg(a)$ and $\sth(b,c)$. We also have
the grading by genus of the underlying
surface, which we again write as a subscript.

We set $\Fat\Sttree(n):=\colim_{\calS} \Fat\Gtree(n)$.
Notice that  $\Fat\St\Gtree(0)$ is again a point. It can be given as $([D],[1])$
where $[D]$ is the image of the disc under stabilization and $[1]$ is
the orbit of the constant map $\gap(*)=1$ under the $\Rp$ action.

In order for the operad structure to descend, we will need structure lemmata
as before.
\begin{lem}
\label{fatlemone}
For every $\t\in \Fat\Gtree_0(1))$
and $\a \Fat\LGtree(1)$:  $\t\circ_1\a=\a\circ \t$.
\end{lem}
\begin{proof}
For this we first thicken the edges $e$ by adding the weight $\gap(e)$
 and we also add $\gap(*)$ to the first edge. We think of the new leaves
in the bands as ending on the respective gaps.
Notice that since we are in $\Fat\LGtree(1)$
the order of the edges is preserved and all edges run from $0$ to $1$. Hence
we have a homeomorphism between the two windows of the surface.
 Now we cut off a cylinder
on the boundary $1$.
We fix the the new marked point to be the translate
of the base point along the foliation of the old boundary $0$. After
cutting off the cylinder, we create the gaps by deleting the
leaves which used to end on the gaps, before thickening them.
\end{proof}
\begin{lem}
\label{fatlemtwo}
Any element $\a$ in $\Fat\Gtree(n)$ can be written as
\begin{equation}
\label{fatdecompeq} \a=T_{a_0}\circ_1 \G\circ_1 T_{a_1} \circ_1\dots \circ_1
\G \circ_1 T_{a_k}\circ_1 \H_{b_1}\circ_1 T_{a_{k+1}} \dots\circ_1
\H_{b_l}\circ_1 T_{a_{k+l}} \circ_1\t\circ_1 \a'
\end{equation}
with  $\a'\in \Gtree(n)$ quasi--filling and not twisted at $0$,
$\tau\in \Fat\Gtree_0(1)$.
\end{lem}
\begin{proof}
Any element $\a$ of $\Fat\Gtree(n)$ can be  decomposed as
$\a=\tau\circ_1\a'$ with $\tau\in \Fat\Gtree_0(1)$ and $\a'\in
\Gtree(n)$ by simply cutting off a small annulus around the boundary
$0$. Now we can decompose $\a'$ according to equation
(\ref{decompeq}) and using Lemma \ref{fatlemone} obtain the
decomposition above.
\end{proof}

\begin{prop}
The image of an element $\a\in \Fat\Gtree$ is given as
follows. Let $\a$ be decomposed as in equation (\ref{fatdecompeq})
then, $[\a]=[T_b\circ_1\t\circ_1 \a']\in \Sttree$ with $b$ the sum of all the
twists.

If $\a=\a_1\circ_1 \t\circ \a'$ and $\b=\b_1\circ \t' \circ \b'$ as in Lemma
\ref{fatdecompeq} where we aggregated all the $T,G,H$ terms into
$\a_1$ and $\b_1$ then $\a\circ_i \b = \g_1 \circ_1\t'' \circ_1 \g'$ in the
same notation with $\t''\circ_1 \g'= \t\circ_1 \a'\circ_i  T_b\circ_i\t' \circ_i \b'$ where $b$ is
the sum of all the twists in $\b_1$.
\end{prop}

\begin{proof}
In view of Lemma \ref{fatlemone} and Lemma \ref{fatlemtwo} the first
part is analogous to Proposition \ref{stimprop} and the second part
is analogous to Proposition
 \ref{structureprop}.
\end{proof}

We define  $\Fat\LGtree$ and  $\Fat\LGtree'$ analogously to their
non--thickened counterparts. In view of the Proposition we obtain:
\begin{thm}
The operad structure of $\Fat\Gtree$ descends to $\Fat\St\Gtree$
and $\Fat\St\LGtree$ and $\Fat\St\LGtree'$ are suboperads.
\end{thm}
\qed

\section{The $E_k$ and $E_{\infty}$ operad structures}

\subsection{Berger's Complete graphs operad}
Since we will use the Fiedorowicz--Berger criterion for $E_{\infty}$
and $E_n$ operads \cite{Berger,Fiedobscure}, we quickly recall the
necessary definitions for  the complete graph poset $\K$. We set
$\K_p:=\N^{p\choose 2}\times {\mathbb S}_p$ and think of an element
$(\mu,\sigma)$ as a collection of natural numbers $(\mu_{ij})_{1\leq
i<j<\leq p}$ and a permutation $\sigma$ in the symmetric group on
$p$ letters ${\mathbb S}_p$.

The sets $\K_{p}$ form an operad under the compositions
\begin{equation}
\Gamma((\mu,\sigma);(\mu_1,\sigma_1),\dots (\mu_p,\sigma_p))=(\mu(\mu_1,\dots,\mu_p),\sigma(\sigma_1,\dotsm,\sigma_p))
\end{equation}
where $\sigma(\sigma_1,\dotsm,\sigma_p)$ is the usual block
permutation and $(\mu(\mu_1,\dots,\mu_p))$ is defined as follows: if
$i,j$ are in the same block say $r$ then $(\mu_r)_{ij}$ is kept, if
they belong to different blocks,  say $r$ and $s$ then one takes
$\mu_{rs}$ keeping in mind the usual renumbering; see \cite{Berger}
for more details. The neutral element is $*=(\emptyset,\emptyset)$
and the map $\phi^*_{ij}$ in this case is given by
$\phi^*_{ij}(\mu,\sigma)=(\phi^*_{ij}(\mu),\phi^*_{ij}(\sigma))$ and
$\phi^*_{ij}(\sigma)$ is the restriction of the permutation to $i$
and $j$ where $i$ is mapped to $1$ and $j$ to $2$ and

\begin{equation}
\phi^*_{ij}(\mu)=\begin{cases} \mu_{ij}& \text{ if } i <j\\
\mu_{ji} &\text{ if } j<i\\
       \end{cases}
\end{equation}

Each set $\K_{p}$  is a poset with the order given by
\begin{equation}
(\mu,\sigma)\leq (\nu,\tau) \Leftrightarrow \forall i<j \text{
either } \phi_{ij}^*(\mu,\sigma)=\phi_{ij}^* (\nu,\tau) \text { or }
\mu_{ij} <\nu_{ij}
\end{equation}

\subsection{Identifying the $E_n$ and $E_{\infty}$ operads}

We recall the following definition from \cite{Berger}:
\begin{df} Let $A$ be a partially ordered set and $X$ a topological space. A collection
 $(c_{\a})_{\a\in A}$ of closed contractible subspaces (so-called ``cells'') of $X$ are called a {\em cellular $A$ decomposition} if the following three conditions hold:
\begin{enumerate}
\item $c_{\a}\subseteq c_{\b} \Leftrightarrow \a\leq \b$
\item the cell inclusions are cofibrations
\item $X=\varinjlim x_{\a}$, so $X$
equals the union of the cells and the weak topology with respect to its cells.
\end{enumerate}

\end{df}

\begin{lem}
 $\St\LGtree(2)$ has a cellular  ${\mathcal K}_2$ decomposition.
\end{lem}

 We use the term cellular here in the sense of Berger.
Below, we also give a CW model where the ``cells'' are actual cells. In
fact  $\St\LGtree(2)$ is homeomorphic to $S^{\infty}\times {\mathbb R}_{>0}$
with the ${\mathcal K}_2$ being the hemispherical decomposition of
Fiedorowicz and Berger \cite{Fiedobscure, Berger} on the factor
$S^{\infty}$ and trivial in the $\Rp$ factors. The cellular model
will have cells exactly corresponding to the hemispherical decomposition.

\begin{proof}
By Corollary \ref{quasicor} the elements of $\St\LGtree(2)$ can be
identified with the elements of $\LGtree(2)$  that are
quasi--filling. It is straightforward
 to verify that these
elements are either of one of the two types of Figure \ref{cup2b} or
with their labels $1$ and $2$ interchanged. We will call these
graphs $\cup_i$ and $\t_{12}\cup_i$. This means that the elements
are indexed by the arc graphs which in turn can be indexed by
$\N\times {\mathbb S}_2$. Here in a tuple $(i,\sigma)\in N\times
{\mathbb S}_2 $
 $i$ is the dimension, which
is the total number of arcs  of $\cup_i$ minus two and $\sigma$ is
either $id$ or $\t_{12}$.

The spaces of elements of a fixed graph are contractible. For this
we just decrease all the weights except on each first arc at each
boundary to zero and then  scale the remaining weights to the same
value $\frac{1}{i+2}$. Now, the codimension $1$ boundary strata
 of each cell are given by deleting one arc.
We see from the alternating structure of the arcs that the result
will be of codimension $2$ unless we are deleting the first or last
arc. In all other cases the element becomes unstable and after
removing the degeneracy we are left with two parallel arcs which are
combined into one arc hence decreasing the dimension. Contracting
the first arc or the last arc, we obtain  $\cup_{i-1}$ or
$\t_{12}\cup_{i-1}$.

We can actually say a little more since
$\LGtree(2)=\LGtree^1(2)\times {\mathbb R}_{>0}^2$; see
\S\ref{normalizedpar} below. The subspaces of $\LGtree^1(2)$ indexed by
 $\cup_i$ are just $\Delta^{[(i+1)/2]}\times \Delta^{[i/2]}\sim_{\rm
 homeo}
B^{i}$ and the boundary maps glue the $B^i$ to the $S^{i-1}$
decomposed into  two hemispheres. It is now straightforward
to see that these cell inclusions are cofibrations and the topology
is the weak, induced one.
\end{proof}

\begin{figure}
\epsfxsize = .8\textwidth \epsfbox{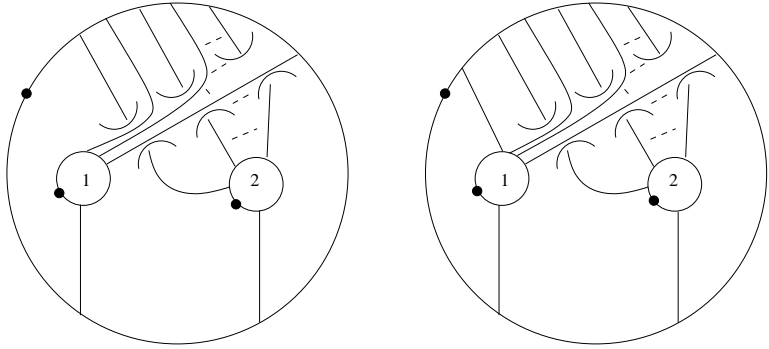}
\caption{\label{cupi} The $\cup_i$ operations for $i$ even and $i$ odd}
\end{figure}

\begin{figure}
\epsfxsize = \textwidth \epsfbox{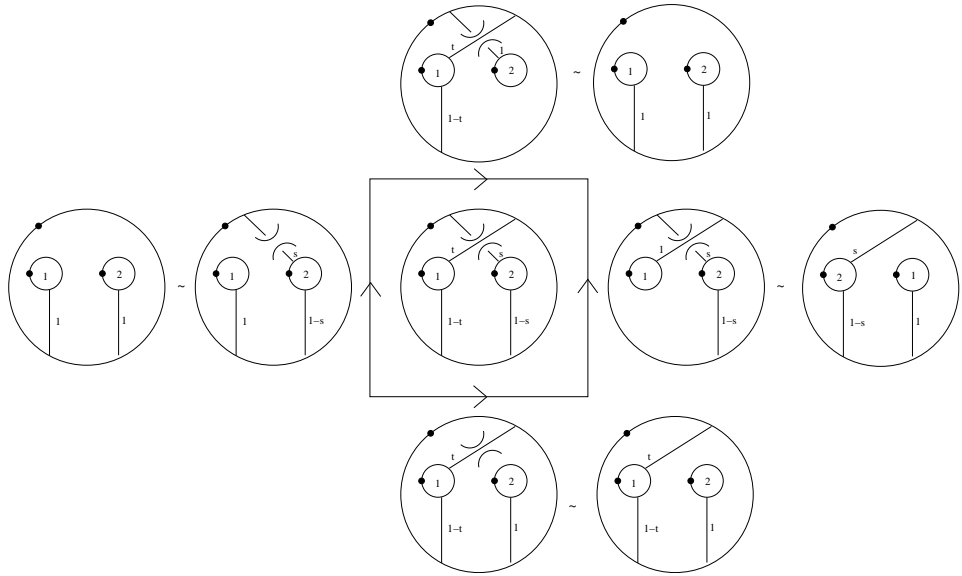}
\caption{\label{cup2} The $\cup_2$ operation and its boundary before
stabilization and stable representatives of the boundary components}
\end{figure}

\begin{figure}
\epsfxsize = .8\textwidth \epsfbox{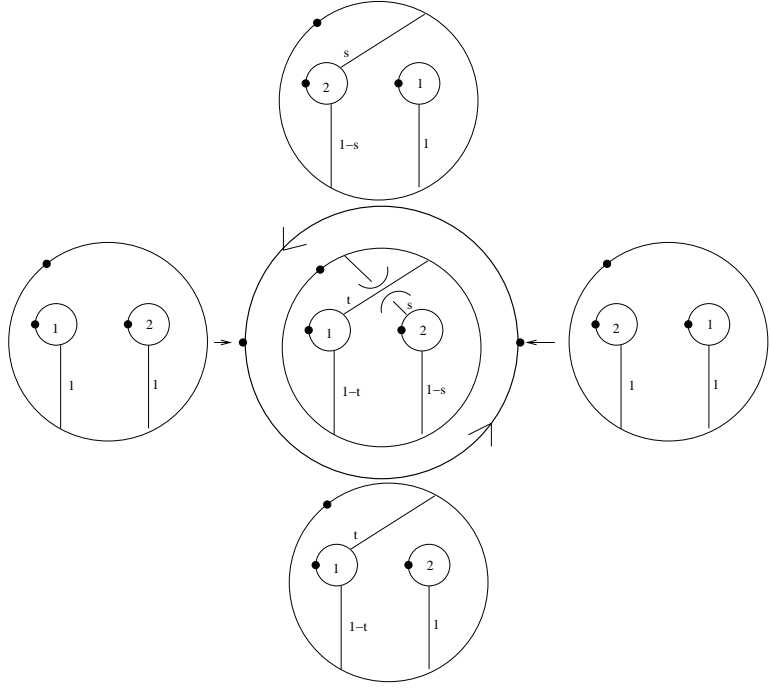}
\caption{\label{cup2b} The $\cup_2$ operation after stabilization}
\end{figure}

\begin{cor}
 $\Fat\St\LGtree(2)$ has a cellular  ${\mathcal K}_2$ decomposition.
\end{cor}
\begin{proof}
Straightforward by contracting the gaps.
\end{proof}

We again repeat definitions of \cite{Berger} extending them slightly to our setting.

\begin{df}
A a topological preoperad ${\mathcal O}$ is called a cellular $E_{\infty}$ preoperad
if the ${\mathbb S}_2$ spaces ${\mathcal O}(2)$ admits a cellular $\K_2$ decomposition
$( \calO^{(\a)}(2))_{\a\in \K_2}$
compatible with the action of ${\mathbb S}_2$, such that
\begin{enumerate}
\item For each $p>0$ and $\a\in \K_p$
$$
\calO(p)^{(\a)}:=\bigcap_{1\leq i<j\leq p}(\phi_{ij}^*)^{-1}(\calO^{(\phi_{ij}^*(\a))}(2))
$$
is contractible, and for each $\a,\b\in \K_p$ with $\a\leq \b$ the
natural inclusion $\calO^{(\a)}(p)\subseteq \calO^{(\b)}(p)$ is a
cofibration.
\item Each ${\mathbb S}_p$ orbit of $\calO(p)$ contains an ordered point, i.e.\
a point $x\in \calO(p)$ whose projections $\phi^*_{ij}(x)$ belong to the interiors
$\check\calO^{(\mu,id)}(2)$ where $\check c_{\a}=c_{\a}\setminus \bigcup_{\b<\a}c_{\b}$.
\end{enumerate}
\end{df}

Introduce the filtration
\begin{equation}
\calO^{(n)}(p)=\bigcup_{\a\in \K_p^{(n)}}\calO^{(\a)}(p)
\end{equation}

\begin{df}
A (weak) unital operad $\calO$ is called a (weak) cellular
$E_{\infty}$ operad if the underlying preoperad is a cellular
$E_{\infty}$  preoperad such that the operadic composition $\Gamma$
satisfies:
\begin{equation}
\Gamma:\calO^{(\mu,\sigma)}(p)\times \calO^{(\mu_1,\sigma_1)}(i_1)\times \dots \times
\calO^{(\mu_p,\sigma_p)}(i_p)\subseteq \calO^{(\mu(\mu_1,\dots\mu_p),\sigma(\sigma_1,\dots \sigma_p)}(\sum_j i_j)
\end{equation}
The suboperads $\calO^{(n)}$ are called (weak) cellular $E_n$ operads.
\end{df}

We will use the weak unital operad $\St\LGtree$ and the unital operad
$\Fat\St\LGtree$.

\begin{prop}
$\St\LGtree$ is a weak cellular $E_{\infty}$ operad and the\\
$\St\LGtree^{(n)}=\{\St\LGtree^{(n)}(p)\}$
are weak cellular $E_n$ operads.

$\Fat\St\LGtree$ is a  cellular $E_{\infty}$ operad and the\\
$\Fat\St\LGtree^{(n)}:=\{\St\LGtree^{(n)}(p)\}$
are  cellular $E_n$ operads.
\end{prop}

\begin{proof}
We have already shown that $\St\LGtree(2)$ admits a cellular $\K_{2}$
decomposition. We will now verify the rest of the conditions.

$\St\LGtree_p^{(\a)}(n)$ is contractible. This is analogous to the
case with $n=2$, we can decrease the weights of the arcs to zero of
all but the first arc in a given boundary component, while at the
same time scaling the weights of the first arcs to the same value.
The fact that the ``cell'' inclusions are cofibrations is again
analogous to the case $n=2$.

The operad multiplication preserves the cellular structure. We see
that gluing in arc families and then discs everywhere, there are
only two situations that can arise. Either the two boundaries were
in the same family  or in two different ones. In both cases we see
that the result of gluing in discs is the same whether we do it
before or after gluing. In the first case this means that we only
look at the surface where both boundaries that are the pre-images of
$i,j$ lie, and in the second case we look at the surface into which
we glue and only keep the boundaries $r$ and $s$ into which the two
surfaces are glued. This is exactly how the composition in $\K$ was
defined.

There is an ordered point in each cell. These points are given by
the iterated gluings $\cup'_n \circ_2 \cup'_n \circ_3  \dots
\circ_{p-1}\cup'_n$ and their images under $\sigma\in \calS_p$,
where $\cup'_n$ is the element in $\cup_n$ whose weights on the arcs
at each of the boundaries are all equal, i.e.
$\frac{1}{[(n+1)/2]+1}$ at the boundary 1 and $\frac{1}{[n/2]+1}$ at
the boundary 2.

The claims about $\Fat\St\LGtree$ as cellular
$E_{\infty}$ preoperad again follow by contracting
$\Fat\St\LGtree(n)$ to $\St\LGtree(n)$ by contracting the gaps.
The cellular $E_{\infty}$ operad structure follows from the above
argument, since the gaps can be ignored for the filtrations.
\end{proof}

\begin{thm}
The operads $\Fat\St\LGtree_k$ are $E_k$ operads and the operad
$\Fat\St\LGtree$ is
an $E_{\infty}$ operad.
\end{thm}

\begin{proof}
Immediate from Fiedorowicz's theorem \cite[Theorem 1.16]{Berger}, which
states that cellular $E_k$ operads are $E_k$ operads and cellular $E_{\infty}$ operads are $E_{\infty}$ operads.
\end{proof}

Using Corollary \ref{equivcor}:
\begin{cor}
The operads  $\{\St\LGtree_k(n),n>0\}$ are equivalent to $\{C_k(n),n>0\}$
where $C_k$ are the little $k$ cubes and the operad \\
$\{\Fat\St\LGtree,n>0\}$ is an $E_{\infty}$ operad without a $0$--term.
\end{cor}

\section{CW models and explicit operations}
\label{chainpar}
There are CW models for $\Gtree$ with the suboperads
being sub--CW--models. Moreover the cellular chains are also a model
for $\Fat\Gtree$, so that if we are interested
only in  the chain level, we can omit the thickening
step. The main point is that even if we include the $0$--th spaces
which makes the topological level only associative up to homotopy,
the structures are already operads on the chain level. Of course this is true
for the homology level {\it a priori} and {\it a forteriori}.
These CW--models are well behaved under the stabilization process and hence
we obtain a combinatorial graph chain model for $\St\Gtree$ and $\Fat\St\Gtree$.

\subsection{The CW--model $\Gtree^1$}
\label{normalizedpar}
 This construction is completely analogous to that of
$\Cacti^1$ of \cite{cact}. For an arc graph $\g$ which belongs to
$\Gtree(n)$ let $v_0$ be the vertex corresponding to the boundary
component $0$. We define $\cell(\g):=\prod_{v\in V(\g)\setminus
\{v_0\}} \Delta^{\val{v}-1}$. We identify the interior of this cell
with all the elements of $\calD\Gtree(n)$ whose arc graph is $\g$
and whose weights at all the boundaries except $0$ are $1$, the
coordinates in $\Delta^{\val{v}-1}$ being given by the barycentric
coordinates corresponding to the weights of the incident arcs in
their order. We identify the boundary of this cell by letting the
weight of the arc corresponding to a vertex of the simplex go to $0$
and erasing it when passing to the face.

We define $\Gtree^1(n)$ to be the CW complex formed from the cells
with the attaching maps given above.
We can add a $0$--th term to the quasi operad $\Gtree^1$ where
$\Gtree^1(0)=\Gtree^1(0)$.

\subsection{The quasi--operad $\Gtree^1$ and its induced cellular operad}
We define the operations $\a \circ_i \b$ on $\Gtree^1$
 by using the alternative
gluing that we scale the weights on the boundary $i$ of $\a$ to match
those of $\b$ at $0$ and then glue in.

\begin{prop}
With the gluings above and the action of ${\mathbb S}_n$ permuting
the labels the spaces $\{\Gtree^1(n)\},n\geq0$ form a homotopy associative
operad (aka. topological quasi--operad) such that
\begin{itemize}
\item[i)] the induced quasi--operad
structure on the cellular chain complex $CC_*(\Gtree^1(n))$ is an operad structure and
\item[ii)] the induced operad structure on $H_*(\Gtree^1(n))$
is isomorphic to  $H_*(\Gtree(n))$

\item[iii)] $CC_*(\Gtree^1(n))$ is a CW model for $\Fat\Gtree$.
\end{itemize}

The subspace $\LGtree^1$ given by the cells with graphs of $\LGtree$ is a
sub--quasi operad on the topological level and a suboperad on
the cellular and homology level.

As spaces $\calD\Gtree(n)=\Gtree^1(n)\times {\mathbb R}_{>0}^n$ and as
quasi--operads $\calD\Gtree(n)=\Gtree^1(n)\rtimes {\mathbb R}_{>0}^n$.
Modding out by the overall $\Rp$ action, we obtain similar results
for $\Gtree$.
\end{prop}

For the definition of semi--direct products of operads see \cite{SW}
and for quasi--operads \cite{cact}.

\begin{proof}
We do not wish to go into the gory details.
The proof is a straightforward adaption
from that of $\Cacti$ presented
in \cite{cact,del}.
For the semi--direct product the first homeomorphism
is given by reading off the weights at each boundary and then taking
the projective class of the weights at each boundary individually.
The semi--direct product is given by first scaling, then inserting and
finally scaling back.

The statement about $H_*(\Fat\Gtree)$
follows from Proposition \ref{retractprop}.
\end{proof}

\begin{thm}
The cellular models carry over to the stabilized situation. In
particular $\St\Gtree$ and $\St\LGtree$ have operadic CW--models
$\St\Gtree^1$ and $\St\LGtree^1$ whose underlying spaces are topological--quasi
operads. In the former, the cells are indexed by the quasi--filling
arc graphs, while in the latter the graphs also satisfy the
conditions of $\LGtree$. Furthermore adding  $0$--components
the analogous statements
hold true and the cellular chains of $\St\Gtree^1$ are a
model for $\Fat\St\Gtree$.
\end{thm}
\begin{proof}
The indexing by quasi--filling graphs is clear in view of Corollary
\ref{quasicor}. The rest is straightforward using the usual
techniques of \cite{cact,hoch1}.
\end{proof}

\begin{rmk}
We wish to point out that the boundary of a cell $\cell(\g)$ now
consists of those graphs obtained by removing an arc from
 $\g$ which is not the only arc incident to a boundary and then stabilizing.
For an illustration see Figure \ref{cup2} and consult the Appendix
for further remarks and discussion.
\end{rmk}

\subsection{Explicit operations}
In view of the CW--models above we can easily write down cellular
representative for classical operations.

\subsubsection{The $\cup_i$ operations}

We have already identified the $\cup_i$ products in the
hemispherical decomposition. The graphs are in Figure \ref{cupi}.
Now we can take the same graphs and reinterpret them as generators
of $CC_*(\St\Gtree^1(2))$.

\begin{rmk}
We wish to point out that these operations belong to the appropriate
part of the filtration of the cellular $E_{\infty}$ operad.
More interestingly, the $\cup_i$ product is realized on a surface of genus $g=[i/2]$.

Moreover this periodicity also manifests itself in the fact that
$\cup_{2k}$ is in the image of $\LGtree'$ and $\cup_{2k+1}$ is
twisted at zero. This means that the sequence is: twist at zero, add
genus, twist at zero etc..

\end{rmk}

\subsubsection{Dyer--Lashof operations}
In this formalism we can also make the Dyer-Lashof-Cohen operations
for double loop spaces explicit. By the general theory, see
\cite{Co} we need to find particular elements
\begin{equation}
\label{xieq} \xi_1 \in H_{p-1}(C_2(p)/{\mathbb S}_p,\pm \Z/p\Z)
\end{equation}
 that is
homology classes of the little 2--cubes
with values in the sign representation.

Now taking co-invariants on $CC_*(\St\LGtree^1(p))=CC_*(\Cact^1(p))$
which is a chain model for the little discs operad $D_2(p)$ we
see that the $p$--the iteration of
the product $\cup_1$ that is the operation given by
\begin{equation}
{}^p\cup_1:=\gamma(\gamma(\dots (\gamma(\cup_1),\cup_1),\dots,\cup_1),\cup_1 )
\end{equation}
gives a class that is the sum over all trees of the highest
dimension where the partial order on the labeled vertices when
considered in the usual tree partial order is compatible with the
linear order on $\bar n$.

\begin{prop}
${}^p\cup_1$ represents
 the cohomology class $\xi_1$ of equation (\ref{xieq})  in
$H_{p-1}(\St\LGtree_2(p)/{\mathbb S}_p,\pm \Z/p\Z)$.
\end{prop}
\begin{proof}
This is a tedious but fairly  straightforward
calculation  of the boundary of said combination of cells.
The actual calculation can be adapted from the proof of Tourtchine
 \cite{tourch} using cells instead of operations on the Hochschild complex.
 The dictionary for this is provided by \cite{del}.
\end{proof}

The first example for $p=2$ is  given by the operation of $\cup_1$
which has boundaries in the multiplication and its opposite, cf.\
Figure \ref{cupi}, and the example for $p=3$ is the hexagon of
Figure \ref{dl3} with $i=1$.

\begin{figure}
\epsfxsize = \textwidth \epsfbox{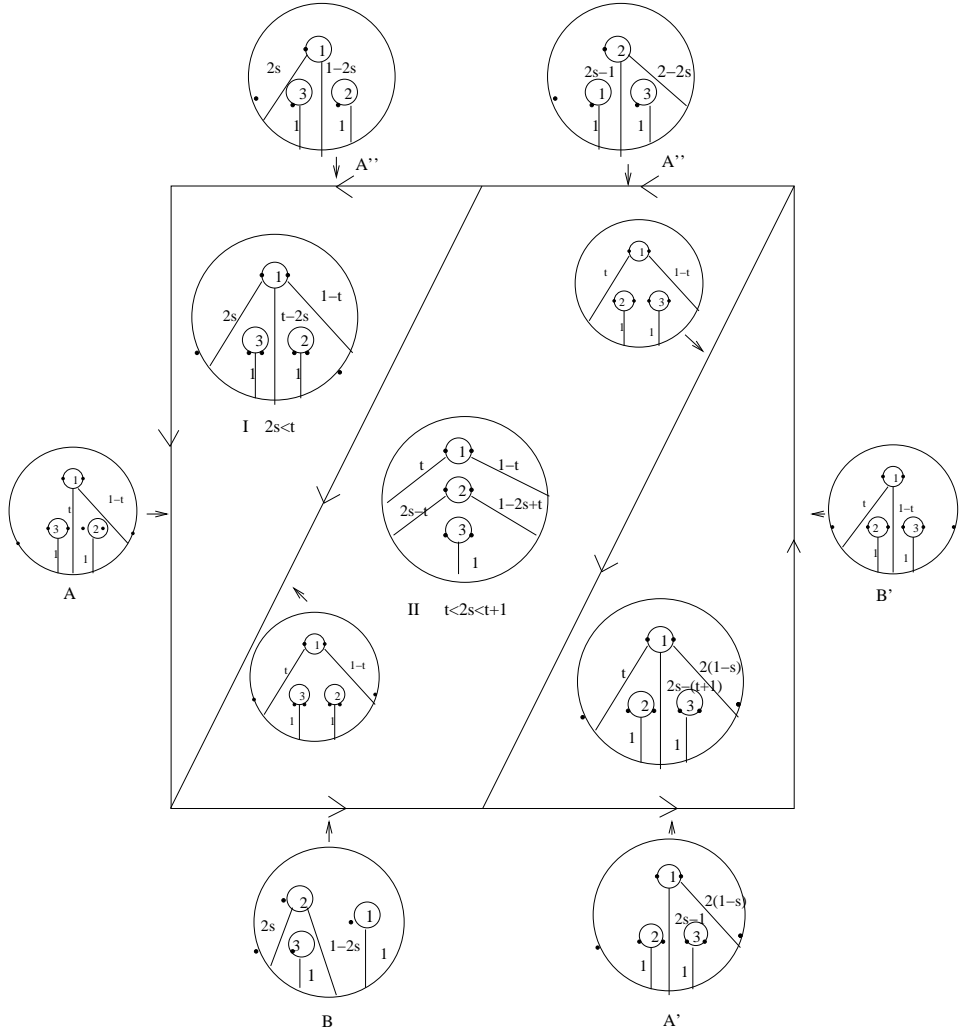} \caption{\label{dl3} The
hexagon that gives the Dyer--Lashof operation. The cells A,A' and
A'' become equivalent as well as B, B' and B'' become equivalent
after passing to coinvariants using the sign representation.}
\end{figure}

\begin{rmk}
We wish to point out several interesting facts.
\begin{enumerate}
\item The class is
solely induced by an operation for $p=2$.
\item The resulting
cell description is just the left iteration of $\cup_1$, whereas the
right iteration of $\cup_1$ is the simple class given by a cube.
\item When acting on $p$ times the same even element, the
two iterations coincide.
Furthermore the action factors through the coinvariants.
When acting on $p$ times the same  odd element, the action
factors through the coinvariants under sign representation.
This is  not surprising, but here we have a very geometric picture.

For the actions, we can take the action of the chain operad on itself of the action on the Hochschild complex
as defined in \cite{del}. Using the former action, we obtain a universal geometric point of view.
\end{enumerate}
\end{rmk}

\begin{rmk}
Using the induced action on the Hochschild complex of \cite{del}, we
 reproduce the results of
Westerland \cite{craig} and Tourtchine \cite{tourch} on representing the
Dyer--Lashof--Cohen operations on the Hochschild co--chains.
\end{rmk}

\subsection{Relation to the McClure--Smith sequence operad}

\begin{df} The {\em sequence} of an element $\a\in \GTree(n)$ is the sequence
$\seq(\a):{1,\dots, E_{\Gamma(\a)}}\to \{1,\dots,n\}$,
defined by the identification of $(E_{\Gamma(\a)},\prec_0)$
with ${1,\dots, E_{\Gamma(\a)}}$
provided by $\prec_0$ and the
 map $(E_{\Gamma(\a)},\prec_0)\to \{1,\dots,n\}$ which maps
each element of the ordered set $(E_{\Gamma(\a)},\prec_0)$ to the label of
its incident boundary that is not the boundary labeled by $0$.
\end{df}

\begin{rmk}
In general for $\Gtree$ the information $\seq(\a)$ does not contain all the information about $\a$. This is different for $\St\LGtree$.
\end{rmk}

We already know that both
the sequence operad of McClure and Smith \cite{McS}, which we will
call $\MS$, and the
operad $\St\LGtree$ are $E_{\infty}$ operads
 and hence they  are equivalent.
The map above makes this
explicit.

\begin{prop}
The map $\a\mapsto \seq(\a)$ induces a surjective morphism of
operads $\St\LGtree \to \MS$. It maps the cellular filtration above
to the filtration by complexity which was introduced in \cite{McS}.
\end{prop}

\begin{proof}
The fact that this is an operadic morphism that is surjective is
only an unraveling of the definitions. Also it is straightforward
from the definitions that the complexity of the sequence of boundary
arcs corresponds exactly to the filtration induced by the
$\K$--structure.
\end{proof}

\begin{rmk}
Probably the operads are even isomorphic, but for our purpose to make the
equivalence explicit
the result above is sufficient.
\end{rmk}

\section{Applications and Outlook}

\subsection{Actions on Hochschild}

Let $A$ be a commutative Frobenius algebra with non--degenerate
pairing $\langle \;,\;\rangle$ and unit $1$. Set $\int a:=\langle
a,1\rangle$. Let $e$ be the Euler element of $A$. Let $\Delta$
be the comultiplication which is the adjoint of the multiplication $\mu$,
then $e:=\mu\Delta(1)$.

\begin{prop}
The action of $CC_*(\Gtree)$ on the Hochschild co--chains
$HC^*(A,A)$ as defined by restricting the action
\cite{hoch1,hoch2} passes to $CC_*(\Sttree)$ if and only if the
Euler element of $A$ is the unit.
\end{prop}

\begin{proof}
We would only like to recall that the action is given by the product
over local contributions where in particular there is one such
contribution for each complementary region which are of the form
$\int a_1 \dots a_n e^{-\chi(R)+1}$, where $e$ is the Euler element
of $A$. The effect of stabilizing is to set the factor of
$e^{\chi(R)-1}$ to $1$, whence the claim.
\end{proof}

\begin{ex}
The condition is met in the case that $A$ is
semi-simple with a unital metric. I.e.\ there is a basis of
idempotents $e_ie_j=\delta_{ij}e_i$ and $\int e_i=1$.
\end{ex}

\begin{rmk}
We again wish to make several remarks
\begin{enumerate}
\item In particular, if $A=H^*(X)$ with $X$ a compact manifold, $A$
is semi--simple only  if $X$ is a point. In all other cases the only
way would be to formally invert the nilpotent element $e$ which
yields the zero algebra. This means that the stabilization is not
compatible with string topology, which is good, since otherwise the
string bracket would vanish.
\item  We know that for Fano varieties, which have  a system of
exceptional sheaves of appropriate length,
the quantum cohomology is however semi--simple. This points to a
connection with quantum cohomology.

\item We expect that the stabilization in our sense is related to the
stabilization in the usual sense (see e.g.\ \cite{Harer}), see below. In  particular this
gives a point of contact with  Witten's $\tau$ function and higher
Weil--Petersson volumes \cite{KMZ,MZ}.

\item In order to obtain  actions in a wider setting
one could try to use  a conformal scaling or to alter the
differential.
\end{enumerate}
\end{rmk}

\subsection{$\Omega$ spectra}
Since $\Fat\St\LGtree$ is an $E_{\infty}$ operad it detects infinite
loop spaces and so does any operad it is a suboperad of.
Furthermore, since it acts on any operad, of which
it is a suboperad, the
respective operad will yield an $\Omega$ spectrum.

\begin{thm}
The group completions of  $\amalg_n \Fat\St\LGtree(n)$,
$\amalg_n \Fat\St\CGtree(n)$ and
 $\amalg_n \Fat\Sttree(n)$ are infinite loop
spaces and hence yield $\Omega$ spectra. Furthermore
all the operads detect infinite loop spaces.
\end{thm}
\qed

\subsection{Stabilized arcs and Sullivan PROP}

\subsubsection{Stabilizing the Sullivan PROP of \cite{hoch1}}
Of course $\Gtree$ is also a suboperad of $\Arc$ and of the Sullivan
PROP of \cite{hoch1}. It would hence make sense to try and stabilize
these two constructions. For the Sullivan PROP this is  rather
straightforward, since we again can choose to stabilize at the out
boundaries only.

Without going into the full details, we define the stabilized Sullivan
quasi--PROP
to be the colimit over all maps in the system $\calS$ where now we are
allowed to glue to any out boundary.

\subsubsection{Stabilizing the $\Arc$ operad.}
For the $\Arc$ operad we have to use a cyclic
alternative.  That is as a first approximation we would like to define the stabilization
by gluing on the system $\calS$ in all possible ways. This poses no problem.
But, we will also have to deal with other types of degeneracies, which lead to
disconnected graphs; see the Appendix for some more details.

 Using this more careful analysis it will be possible to add a neutral element
using thickenings.

 \begin{conj}
 There is a suitable stabilization $\Fat\St\Arc$ of a
thickening of $\Arc$ whose $0$--term is the disc and which contains $\Fat\St\Gtree$
as a suboperad and hence the  group completion
of $\amalg_{n\geq 0} \Fat\St\Arc(n)$ is an infinite loop space.
 \end{conj}

 It will be interesting to figure out which $\Omega$ spectrum this is.
 Since stabilization by the element $\G$
 embeds the moduli space $M^{1^{n}}_{g,n}$
 as a piece of
  the boundary of the moduli space $M^{1^{n}}_{g+1,n}$, we expect that it will
 be closely related to the Segal--Tillmann picture \cite{Til1,Til2}.

\subsection{Outlook: Generalizing framed little discs and new decompositions}
One open question is the full role of the spaces $\Gtree_g(1)$ that
have been discussed below. In the case of genus one $\Arc_0(1)$ made
the difference between framed little discs and little discs.
Similarly, $\Arc_0(1)$ leads to a bi--crossed product $\CGtree$. One
could wonder what the inclusion the space $\Gtree_g(1)$ signifies.

We can decode some of its structure.

\begin{prop}
An element $\a$ of $\Gtree_g(1)$ which is untwisted has at most
$3g+1+[(g-1)/3]$ arcs and this number is realized.
\end{prop}

\begin{proof}
We begin cutting along the arcs. The maximal number of
non--separating cuts is $2g+1$, since the Euler characteristic of
the underlying surface is $2-2g+2$ and each non--separating cut
increases the Euler characteristic by $1$. There are no more
non--separating cuts when we are left with a disc of Euler
characteristic $1$. This disc will have a boundary made up of a
$4(2g+1)$--gon. The sides are labeled by sequences where every 4th
element is a part of the boundary $0$, every $4n+2$ element is a
part of the boundary $1$ and the $4n+1$st and $4n+3$rd elements
correspond to the cut edges --- where each edge appears twice. For
$g>0$, we can insert a maximum of $[4(2g+1)/6]=g+[(g-1)/3]$ arcs,
since each separating arc has to cut off at least an octagon. This
is because each arc has to run from $0$ to $1$ and these cannot be
only separated by one edge since otherwise the new arc and the arc
represented by the edge would be parallel.
\end{proof}

\begin{cor}
The dimension of the top dimensional cells of $\Gtree^1_g(1)$, $g>0$
is $3g+[(g-1)/3]+2$.
\end{cor}
\begin{proof}
For the total count, we can add twists at both ends as soon as
$g>0$. Thus keeping in mind that the dimension is one less than the
number of arcs, we arrive at the formula above.
\end{proof}

Thus $\Gtree$  cannot give the framed little discs as the dimension
only grows linearly in $g$ and not quadratically. On the other hand
the dimensions fit with the dimension of spheres, so that it may
look like a marked point on the boundary of the little cubes.

By Lemma \ref{finerlem}, however, we see that we get a new decomposition
for the $E_n$ operads in terms of $E_1$ respectively $E_2$ and elements
of $\Gtree_g(1)$. Notice though that if we decompose $\cup_2$ we
can decompose it as $\cup_1$ and an element which is not in $\CGtree(1)$.
However this element contains exactly one braid. Furthermore, we
see that the elements $\cup_i$ are generated by $\cup_1$ and particular
braid elements of $\Gtree_{[(i-1)/2]}(1)$.

\begin{conj}
For each $n$ there is a suboperad of $\St\Gtree(1)$ such that the
operad $\St\LGtree^{(n)}$ (an $E_n$ operad without $0$--term) is a
bi--crossed product of this suboperad and operad $\Lintree$ (an
$E_2$ operad without $0$--term).
\end{conj}

\renewcommand{\theequation}{A-\arabic{equation}}
\renewcommand{\thesection}{A}
\setcounter{equation}{0}  
\setcounter{subsection}{0}

\section*{Appendix: Graphs, dual graphs and compactifications}
In \cite{del} (see also \cite{hoch1}) we introduced a dual graph for
quasi--filling arc graphs. Here we extend this notion to all graphs.
One upshot is that we can make contact with Kontsevich's stabilization in
this way.
\subsection{Dual graph}
The dual graph of an arc graph $\a$ is the labeled graph
(semi-stable labeled ribbon graph) $\Gamma(\a)$ whose vertices are the
complementary regions of the arc graph. We will write $v(R)$ for the
vertex corresponding to a region $R$. Edges correspond to the arcs
of the arc graph and the  vertices of an edge to the complementary
region(s) that are bordered by the respective arc. A flag will be a
pair of an arc and a choice of orientation for it or equivalently a
side of the arc. Notice that loops are allowed. Since the surface
$\Sigma$ was connected, the graph will be connected.

There is a bit more structure on these graphs, although they generally fall
short of being ribbon graphs. To make this discussion more symmetric, we will
use one of the equivalent versions for depicting arcs; see \ref{picturepar}.
This is, we move the end--points of the edges off the marked points on the boundary and move
them apart along the orientation of the boundary, so that the arcs are all disjointly embedded, do not hit the endpoints, and their linear order from the marked ribbon graph
coincides with the linear order given by counting them off starting at the marked point
of a boundary and going around that boundary in its induced orientation.
Fix a complementary region $R$. $R$ is an oriented
surface with boundary. There are two types of boundary components, those which contain arcs of the original graph and those who do not. The former
are actually $2n$-gons whose sides alternate between pieces of the boundary
and arcs, while the latter consist of a full boundary component of $\Sigma$.
Let $b(R)$  be the number of boundary components of the
former type, $f(R)$ the boundary components of the latter type and fix
$g(R)$ to be the genus of $R$ after gluing in discs into the boundary.
We set $\defect(R)=(g(R),b(R),f(R))$.

Now each boundary of $R$ containing arcs has an induced orientation, hence
we get a cyclic order on these arcs. Formally this means that at
each vertex $v$ of $\Gamma(\a)$ we have an action of ${\mathbb N}$.
Each orbit corresponds to a set of flags stemming from one of the
boundaries of $R$. In other words the set of flags $F(v)$ is
partitioned into subsets $F(v)=F_1(v)\amalg\dots\amalg F_k(v)$ and
each $F_i(v)$ has a cyclic order. Moreover these cyclic orders fit
together to give an action of $\mathbb N$ on the set of all flags
$F(\Gamma(\alpha))$, by combining the previous action with the map
$\imath$ as usual. The orbits of this map, which we call $N$ are
still called cycles of $\Gamma$.

Lastly there is a marking $mk$ for each cycle of the graph. This is
the first flag of the cycle which corresponds to the flag of the
edge containing the marked point of the boundary.

\begin{df}
The dual graph of $\a$ is defined to be the graph
$(\Gamma,N,\defect,mk)$.
\end{df}

An example of the dual graph is given in Figure \ref{cup2graphs}.
Unlike in the situation of $\Tree_{cp}$ where the dual graphs are
cacti, the advantage of arc graphs may be more obvious.

\begin{figure}
\epsfxsize = 0.8 \textwidth \epsfbox{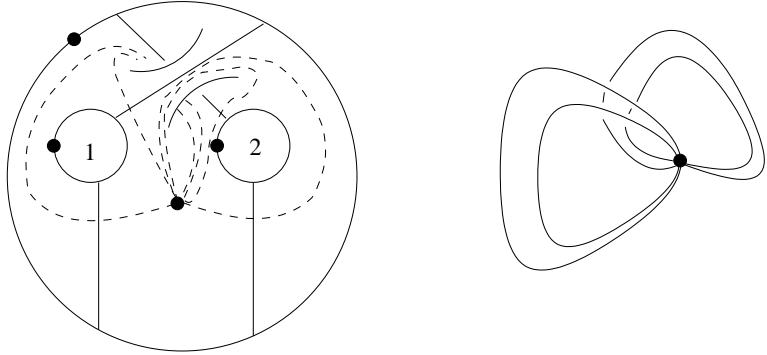}
\caption{\label{cup2graphs} The arc graph for the $\cup_2$ product
and its alternate depiction as a ribbon graph}
\end{figure}

The data of $b(R)$ is actually redundant, since $b(R)$ is the number
of orbits of the cyclic action of $\mathbb N$ on the flags at that
vertex.

\begin{rmk}
 In the case of $\Gtree$ we always have that $f(R)=0$, since there are no boundaries
which are not hit by arcs.
\end{rmk}

\begin{rmk}
The use of the dual graph now gives a re--interpretation of Penner's
compactification \cite{Penner} in terms of Kontsevich's \cite{K1}
and vice--versa.
\end{rmk}
\subsection{Stabilization}
 The stabilization will then have the effect of setting
the label $g(R)$ of a vertex to zero or in the case that $b(R)\neq
0$ the vertex will be split into the number of boundaries.

\begin{rmk}
Although in general the graph can become disconnected this does not
happen for $\Gtree$. The reason is that if it were disconnected,
then there would be a separating curve which does not cut any
of the arcs. This is impossible if all the arcs run to zero.
\end{rmk}

\begin{figure}
\epsfxsize =  \textwidth \epsfbox{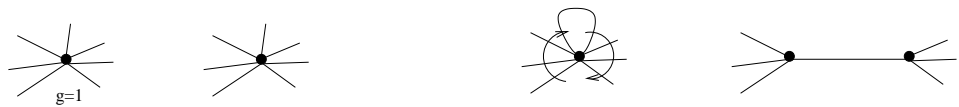}
\caption{\label{graphstable} The stabilization of a vertex using
$\G$ or $\H$ in the case of $\Gtree$.}
\end{figure}

\end{document}